\documentclass[10pt, oneside]{article}
\usepackage{mathrsfs}
\usepackage{amsfonts}
\usepackage{amssymb} 
\usepackage{amsmath,amsfonts,amssymb,amsthm}
\usepackage{caption}
\usepackage{subfigure}
\usepackage{cite}
\usepackage{color}
\usepackage{graphicx}
\usepackage{subfigure}
\usepackage{float}
\usepackage{paralist}
\usepackage{indentfirst}
\usepackage{authblk}

\usepackage[colorlinks,linkcolor=blue,anchorcolor=blue,citecolor=blue,  urlcolor=blue]{hyperref}
\usepackage{indentfirst}
\usepackage{authblk}
\usepackage[T1]{fontenc}
\usepackage{amscd}
\usepackage{latexsym}
\usepackage{verbatim}
\usepackage{enumerate}
\usepackage{fullpage}
\usepackage{xcolor}
\usepackage{titlesec}
\usepackage{titletoc}
\usepackage{nameref}
\usepackage{appendix}
\usepackage{doi}

\def\div{ \hbox{\rm div}\,  }

\newtheorem{theorem}{Theorem}[section]
\newtheorem{lemma}{Lemma}[section]

\newtheorem{remark}{Remark}[section]
\newtheorem{proposition}{Proposition}[section]
\numberwithin{equation}{section}
\begin{document}
\title{\textbf{Global dynamics for the generalized chemotaxis-Navier-Stokes system in $\mathbb{R}^3$}}

\author{Qingyou He, Ling-Yun Shou, and Leyun Wu}

\date{}

\renewcommand*{\Affilfont}{\small\it}
\maketitle
\begin{abstract}
 We consider the chemotaxis-Navier-Stokes system with generalized fluid dissipation in $\mathbb{R}^3$:
 \begin{eqnarray*}
	\begin{cases}
\partial_t n+u\cdot \nabla n=\Delta n-
\nabla \cdot (\chi(c)n \nabla c),\\
\partial_t c+u \cdot \nabla c=\Delta c-nf(c),\\
\partial_t u +u \cdot \nabla u+\nabla P=-(-\Delta)^\alpha u-n\nabla \phi,\\
\nabla \cdot u=0,
	\end{cases}
\end{eqnarray*}
which describes the motion of swimming bacteria or bacillus subtilis suspended to water flows. First, we prove some blow-up criteria of strong solutions to the Cauchy problem, including the Prodi-Serrin type criterion ($\alpha>\frac{3}{4}$) and the Beir${\rm\tilde{a}}$o da Veiga type criterion $(\alpha>\frac{1}{2})$.  Then, we verify the global existence and uniqueness of strong solutions for arbitrarily large initial fluid velocity and bacteria density for $\alpha\geq \frac{5}{4}$.  Furthermore, in the scenario of $\frac{3}{4}<\alpha<\frac{5}{4}$, we establish uniform regularity estimates and  optimal time-decay rates of global  solutions if the $L^2$-norm of initial data is small. To our knowledge, this is the first result concerning the global existence and large-time behavior of strong solutions for the chemotaxis-Navier-Stokes equations with possibly large oscillations. 
 \end{abstract}

 \noindent{\bf Keywords.} Chemotaxis-Navier-Stokes system;  Generalized dissipation; Blow-up criteria; Global existence; Optimal time decay.
\\
2020 {\bf MSC.} 35A01; 35B40; 35B44; 35K55; 92C17.

\section{Introduction}

In this paper, we explore mathematical representations that capture the behavior of oxygen, swimming bacteria, and viscous, incompressible fluids.
We consider the Cauchy problem of the chemotaxis-Navier-Stokes system with generalized fluid dissipation:
\begin{equation}\label{eq1-1}
	\begin{cases}
\partial_t n+u\cdot \nabla n=\Delta n-
\nabla \cdot (\chi(c)n \nabla c),\\
\partial_t c+u \cdot \nabla c=\Delta c-nf(c),\\
\partial_t u +u \cdot \nabla u+\nabla P=-(-\Delta)^\alpha u-n\nabla \phi,\\
\nabla \cdot u=0, 
	\end{cases}
\end{equation}
supplemented with the initial condition
\begin{equation}\label{eq1-2}
	\left(n(t,x), c(t,x), u(t,x)\right)|_{t=0}=\left(n_0(x), c_0(x), u_0(x)\right), \quad  x\in\mathbb{R}^3.
\end{equation}
Here $n=n(t,x): \mathbb{R}^{+}\times\mathbb{R}^3\rightarrow \mathbb{R}^{+}$, $c=c(t,x):\mathbb{R}^{+}\times\mathbb{R}^3\rightarrow \mathbb{R}^+,$ $u=u(t,x):\mathbb{R}^{+}\times\mathbb{R}^3\rightarrow \mathbb{R}^3$  and $P=P(t,x):\mathbb{R}^{+}\times\mathbb{R}^3\rightarrow  \mathbb{R}$, respectively, denote the density of the bacteria (cell) population, the concentration of the oxygen (chemotactic signal), the velocity of the fluid and the pressure.  $f(c)$ and $\chi(c)$ stand for the oxygen consumption rate and the chemotactic sensitivity respectively.  $\phi=\phi(x)$ denotes the  potential function produced by different physical mechanisms.   In addition,  $(-\Delta)^\alpha$ with $\alpha >0$  is the Laplace operator, which is defined by 
$${\widehat {(-\Delta)^\alpha f}}(\xi)=|\xi|^{2\alpha} \widehat f (\xi).$$

  When $\alpha=1$, \eqref{eq1-1} becomes the classical chemotaxis-Navier-Stokes equations.
Usually, one  refers the system 
 to the chemotaxis-Stokes equation when the convection term $u\cdot \nabla u$ is absent in \eqref{eq1-1}$_3$.
The classical chemotaxis-Navier-Stokes equations  was introduced by Tuval et al. \cite{PANS2005} to describe the dynamics of swimming bacteria or bacillus subtilis. According to experiments, the process involves bacteria moving towards areas with a higher concentration of oxygen as a source of energy. The movement of these bacteria also affects the motion of the surrounding fluid, due to their weight. Furthermore, both the bacteria and the oxygen are transported through the water by a convective motion.

\vspace{2mm}

It should be noted that \eqref{eq1-1} with fractional Laplacian operators ($0<\alpha <1$) is physically relevant. Replacing the standard Laplacian operator, these fractional diffusion operators model the so-called anomalous diffusion, 
 a topic that has been extensively explored in physics, probability theory, and finance (see, for instance, \cite{AT2005,Jara2009CPAM,Mellet2011ARMA}).
The deployment of fractional diffusion operators in \eqref{eq1-1} enables the study of long-range diffusive interactions. 
  In addition, \eqref{eq1-1} with hyper-viscosity ($\alpha > 1$) is 
 employed in turbulence modeling to regulate the effective range of non-local dissipation
and to make more efficient numerical resolutions (see, e.g., \cite{FKPPRWZ2008PRL}).


\vspace{2mm}

 Before going into further discussion of the chemotaxis fluid models, we would like to mention that the classical model for cell motion was proposed by Patlak \cite{patlak1953} and Keller and Segel \cite{KS1970,ks1971}. It consists of the dynamics of cell density $n=n(t,x)$ and the concentration of chemical attractant substance $c=c(t,x)$:
\begin{equation}\label{sks}
\begin{cases}
\partial_t n=\Delta n-\nabla \cdot(n \chi(c) \nabla c),\\
\partial_t c=\Delta c-a c+b n,
\end{cases}
\end{equation}
where $\chi(c)$ is the chemotatic sensitivity, and $a$ and $b$ are the decay and production rate of the chemical, respectively. 
The system \eqref{sks}  has been extensively studied by many authors, and a comprehensive list of partial results can be found in 
\cite{wxy2016, arumugam2021, wm2016s,KimYao-2012-SIAM,bookPerthame} and references therein.

\vspace{2mm}

Then, we recall the three-dimensional generalized incompressible Navier-Stokes equations
\begin{equation}\label{ins}
	\begin{cases}
\partial_t u +u \cdot \nabla u+\nabla P=-(-\Delta)^{\alpha} u,\\
\nabla \cdot u=0.
	\end{cases}
\end{equation}
When $\alpha=1$, the system \eqref{ins} reduces to the classical incompressible Navier-Stokes equations. Global weak solutions with finite energy were constructed in the celebrated works by Leary \cite{lery1934} and Hopf \cite{hopf1950}. There are two types of blow-up criteria, namely, for maximal existence time $T^*>0$, the Prodi-Serrin type criterion (see \cite{prodi1959,serrin1961})
\begin{equation*}
\int_0^{T_*}\|u\|_{L^{q}}^{p} \,dt =+\infty
\end{equation*}
with $\frac{2}{p}+\frac{3}{q}\leq 1$ and $3<q\leq \infty$, and the Beir$\tilde{{\rm a}}$o da Veiga type criterion (cf. \cite{veiga1995})
\begin{equation*}
\int_0^{T_*}
\|\nabla u\|_{L^{q}}^{p}\,dt=+\infty
\end{equation*} with  
$\frac{2}{p}+\frac{3}{q}\leq 2 $ and $\frac{3}{2}<q\leq \infty.$
Both of these criteria are important tools for studying the regularity of weak and strong solutions to the Navier-Stokes equations, and have been used extensively in the literature. For more relevant work, interested readers can partially refer to  \cite{pll1996,gs2015,BS,RRS2016} and references therein. In addition, $\alpha=\frac{5}{4}$ is often referred as  Lions' critical exponent.  For $\alpha\geq \frac{5}{4}$, 
the existence of a global classical solution has been established in \cite{jll1969}. For general $\alpha<\frac{5}{4}$, the existence and regularity results can be found in, e.g., 
\cite{TY1, ZhouAIHP, CH1,Wu-2004-DPDE,Wu-2005-CMP}.

\vspace{2mm}

The classical chemotaxis-Navier-Stokes equations, i.e., $\alpha=1$, have been studied extensively with many significant results. Quite a lot of important progress has been made on initial boundary value problems for the chemotaxis-(Navier)-Stokes system, cf. \cite{Lorz2010,ZZ-2021-JDE,wm2023JEMS,wm2012,wm2019,wm2016,wm2017,wm2015,cl2018,lj2016,wx2016,wx2015,wwx2021,tw2016,wy2017,px2019,kz2019} and references therein. Concerning the Cauchy problem, Duan et al. \cite{duan2010} studied the global existence and decay rates of classical solutions over $\mathbb{R}^3$,  provided that the initial datum  $(n_0, c_0, u_0)$ is a small smooth perturbation of the constant state $(n_\infty, 0, 0)$
with $n_\infty \geq 0$ in $H^3(\mathbb{R}^3)$. Moreover, by constructing some proper free energy functional and deriving some uniform a priori estimates, the authors \cite{duan2010} also derived the global existence of weak solutions to the Cauchy problem over $\mathbb{R}^2$. 
Immediately afterwards,  by deriving an entropy, proposing a regularization of the system and a compactness argument to pass to the limit,  Liu and Lorz \cite{liu2011} proved the global existence of  weak solutions for the three-dimensional chemotaxis-Stokes system, i.e., without the convective term $u\cdot\nabla u$. Lorz \cite{lorz2012} showed the global existence of weak solutions  in $\mathbb{R}^3$ with small initial $L^{\frac{3}{2}}$-norm.  
Zhang and Zheng \cite{zz2014} established global well-posedness for the Cauchy problem of the two-dimensional chemotaxis-Navier-Stokes equations by introducing the Zygmund spaces and then establishing the estimates of 
$\|n(t)\|_{Llog L}$ and $\|c(t)\|_{H^1}$. Chae, Kang and Lee \cite{ckl2013} obtained the local existence of regular solutions with $(n_0,u_0)\in H^{m}(\mathbb{R}^d)$ and $c_0\in H^{m+1}(\mathbb{R}^d)$ with $s\geq3$ and $d=2,3$.
In \cite{ckl2014},  they also presented some blow-up criteria and constructed global-in-time solutions for the three-dimensional chemotaxis-Stokes equations under the some smallness conditions of initial data. For the Cauchy problem of the self-consistent chemotaxis-fluid system, Carrillo et al. \cite{cpx2023} established some blow-up criteria of classical solutions and proved the global well-posedness for the two-dimensional 
 case. In addition, they \cite{cpx2023} also obtained the global weak solution with small $c_0$ for the three-dimensional chemotaxis-Stokes flow. 
Recently, Zeng, Zhang and Zi \cite{ZZZ-2021-JFA}  investigated the enhanced dissipation and blow-up suppression for the two-dimensional Patlak-Keller-Segel-Navier-Stokes system near the  Couette flow (also refer to He \cite{He-2023-SIAM} and Hu, Kiselve and Yao \cite{HuKiselevYao1} in different contexts).


When the fractional effect is taken into account, to our knowledge, there are few mathematical results for  chemotaxis-fluid systems so far.  We refer to \cite{LLZ-2019-NA,LLZ-2022-NA} in the case that the fractional diffusion appears in the population density equation.  Concerning the chemotaxis-Navier-Stokes equations \eqref{eq1-1}, Nie and Zheng  \cite{NZ-2020-JDE} studied the global well-posedness in the two-dimensional case where the damping effect of the logistic source was used.  



\vspace{2mm}
However, although considerable progress has been made for the global-in-time evolution of the chemotaxis-Navier-Stokes
equations in three dimensions, all of these results are concerned with weak solutions or strong solutions under the small $H^{s}(\mathbb{R}^3)$-initial data with $s\geq3$. The purpose of  this paper is to 
 develop new blow-up and global existence results for the Cauchy problem \eqref{eq1-1}-\eqref{eq1-2}, which admit a class of large data in $H^{2}(\mathbb{R}^3)$ and especially hold for the case of classical chemotaxis-Navier-Stokes system. 
More precisely, the first result of this paper, i.e., Theorem \ref{Thlocal}, shows the local well-posedness of strong solutions to \eqref{eq1-1}-\eqref{eq1-2} for general initial data. Then, in Theorem \ref{Th1.1}, we study the mechanism of possible finite time blow-up and prove the Prodi-Serrin type criterion for any $\alpha >\frac{3}{4},$ in which the restriction coincides with the regularity theorems for the generalized Navier-Stokes equations. 
Theorem \ref{Th1.2} is devoted to the 
Beir$\tilde{{\rm a}}$o da Veiga type criterion for any $\alpha >\frac{1}{2}$.
Using the criteria established previously, we derive the global well-posedness for a large initial fluid velocity in the case $\alpha \geq \frac{5}{4}$.  
Finally, for the remainder case $\frac{3}{4}<\alpha<\frac{5}{4}$, we develop a new energy argument to control the high-order norms of solutions by some powers of the $L^2$-norm of $(n_0,c_0, u_0)$. This enables us to derive a unique global strong solution to the Cauchy problem \eqref{eq1-1}-\eqref{eq1-2}  and also obtain the optimal time-decay rates if the $L^2$-norm of initial data is small but higher order norms may be arbitrarily large, i.e., Theorem \ref{Th1.3}.

\paragraph{Main results.} Throughout this paper, we assume that the initial data $(n_0, c_0, u_0)$ fulfills 
\begin{equation}\label{H1}
\begin{aligned}
&n_{0},c_{0}\geq 0,\quad\quad \nabla\cdot u_{0}=0. 
\end{aligned}
\end{equation}
With regard to $\chi$, $f$ and $s$, we require 
\begin{equation}\label{H2}
\left\{
\begin{aligned}
&\chi\in W^{2,\infty}_{loc}(\mathbb{R}^{+}),\\
&f\in W^{1,\infty}_{loc}(\mathbb{R}^{+}),\quad f(0)=0,\quad    f(s)\geq0\quad\text{for}\quad s\geq0,\\
&\nabla\phi\in W^{2,\infty}(\mathbb{R}^3).
\end{aligned}
\right.
\end{equation}


 We first give the local well-posedness of strong solutions to the Cauchy problem \eqref{eq1-1}-\eqref{eq1-2}.

\begin{theorem}\label{Thlocal} 
 Let $\alpha>\frac{1}{2}$. Assume that \eqref{H1}-\eqref{H2} hold, and  the initial data    satisfies  $(n_0, c_0, u_0)\newline\in H^{2}(\mathbb{R}^3)$. Then,  there exists a maximal time $T_{*}>0$ such that the Cauchy problem \eqref{eq1-1}-\eqref{eq1-2} admits a unique strong solution  $(n,c,u)$ satisfying 
\begin{eqnarray}\label{localr}
 \begin{cases}
 n,c\in \mathcal{C}([0,T_{*});H^{2}(\mathbb{R}^3))\cap L^2(0,T_{*};\dot{H}^{1}\cap \dot{H}^{3}(\mathbb{R}^3)),\\
 u\in \mathcal{C}([0,T_{*});H^{2}(\mathbb{R}^3))\cap L^2(0,T_{*};\dot{H}^{\alpha}\cap \dot{H}^{2+\alpha}(\mathbb{R}^3)).
 \end{cases}   
\end{eqnarray}
\end{theorem}


Then, we establish some time extensibility criteria of strong solutions with respect to different fractional fluid dissipation exponent. We first state a criterion of Prodi-Serrin type for \eqref{eq1-1}-\eqref{eq1-2}.


  \begin{theorem}\label{Th1.1} 
  Let $ \alpha> \frac{3}{4}$, and the assumptions of Theorem \ref{Thlocal} hold. If the maximal existence time $T_*$ of the solution $(n,c,u)$ to the Cauchy problem \eqref{eq1-1}-\eqref{eq1-2} is finite, then it holds that
\begin{equation}\label{eq1-3}
\int_0^{T_*}\left(\|n\|_{L^{q_1}}^{p_1}+\|u\|_{L^{q_2}}^{p_2} \right) dt =+\infty
\end{equation}
for any $(p_i, q_i) (i=1, 2)$ satisfying 
\begin{align}
\frac{2}{p_1}+\frac{3}{q_1}\leq 1,\quad \frac{2\alpha }{p_2}+\frac{3}{q_2}\leq 2\alpha-1,\quad  3<q_1\leq +\infty ,\quad 
\max\left\{\frac{3}{2},\frac{3}{2\alpha-1}\right\}<q_2\leq \infty.\label{piqi}
\end{align}
\end{theorem}

Next, we also obtain the  blow-up criterion of Beir${\rm\tilde{a}}$o da Veiga type for  \eqref{eq1-1}-\eqref{eq1-2}.

\begin{theorem}\label{Th1.2} 
Let $ \alpha> \frac{1}{2}$, and the assumptions of Theorem \ref{Thlocal} hold. If the maximal existence time $T_*$ of the solution $(n,c,u)$ to the Cauchy problem \eqref{eq1-1}-\eqref{eq1-2} is finite,  it holds that 
\begin{equation}\label{eq1-5}
\int_0^{T_*}\left(
\|\nabla c\|_{L^{q_1}}^{p_1}+
\|\nabla u\|_{L^{q_2}}^{p_2}
\right) dt=+\infty
\end{equation}
for any $(p_i, q_i)$ {\rm(}$i=1,2${\rm)} satisfying 
\begin{equation}\label{piqi1}
    \begin{aligned}
    \frac{2}{p_1}+\frac{3}{q_1}\leq 1, \quad 
\frac{2\alpha}{p_2}+\frac{3}{q_2}\leq 2\alpha,\quad  
3<q_1\leq \infty,\quad \max\left\{1,\frac{3}{2\alpha}\right\}<q_2\leq \infty.    \end{aligned}
\end{equation}

\end{theorem}

\begin{remark}
For the classical chemotaxis-Navier-Stokes equation {\rm(}$\alpha =1${\rm)}, Chae, Kang and Lee {\rm\cite{ckl2014}}
proved the Prodi-Serrin type criterion with  
$p_1=2,\, q_1=\infty$ in \eqref{piqi1}{\rm;} 
Carrillo, Peng and Xiang {\rm\cite{cpx2023}} showed the  a  criterion 
with $\|n\|_{L^{q_1}}^{p_1}$ in \eqref{eq1-3} replaced by 
 $\|\nabla c\|_{L^{q_1}}^{p_1}$. 
Theorem \ref{Th1.1} encompasses the classical case $\alpha =1$ as presented in {\rm\cite{ckl2014}}. 
The restriction $\alpha>\frac{3}{4}$ coincides with the sharp condition for regularity theorems of the generalized Navier-Stokes equations {\rm(}eg. cf. {\rm\cite{TY1,ZhouAIHP,FAHNZ,KimJM,JW1,CH1})}.  
\end{remark}


\begin{remark}
There are some possible extensions for our blow-up criteria. First, if $\frac{5}{6}\leq \alpha \leq \frac{3}{2}$,  then similar computations verify that $\|n\|_{L^{q_1}}^{p_1}$ in \eqref{eq1-3} can be replaced by  $\|\nabla c\|_{L^{q_1}}^{p_1}$. Next, our results can be extended to the chemotaxis-Stokes system. In this case, the criteria \eqref{eq1-3} and \eqref{eq1-5} can be relaxed to
\begin{equation}\nonumber
	\int_0^{T_*}\|n\|_{L^{q_1}}^{p_1}dt =+\infty,
\end{equation}
and
\begin{equation}\nonumber
	\int_0^{T_*}\|\nabla c\|_{L^{q_1}}^{p_1}dt =+\infty,
\end{equation}
respectively. Here $(p_1, q_1) $ satisfies $\frac{2}{p_1}+\frac{3}{q_1}\leq 1\,\, \mbox{and}\,\, 3<q_1\leq +\infty. \,\, $.
\end{remark}



In addition, we state some results about the global-in-time existence and stability. For $\alpha\geq\frac{5}{4}$, we prove the global existence of the strong solution under the  smallness of $\|c_0\|_{L^{\infty}}$. Compared to conventional dissipation $\alpha =1$, 
our research requires only $c_0$ to be small, rather than all initial data $(n_0, c_0, u_0)$. This implies that strong dissipation with 
$\alpha \geq \frac{5}{4}$ has a favorable impact.



\begin{theorem}\label{Th1.3} 
 Let
$\alpha\geq \frac{5}{4}$. Assume that \eqref{H1} and \eqref{H2} hold, and $(n_{0},c_0,u_0)$ satisfies  $(n_0,  c_0, u_0)\newline\in H^{2}(\mathbb{R}^3)$. Additionally, there exists a generic constant $\delta_{0}>0$ such that if 
\begin{align}
\|c_{0}\|_{L^\infty}\leq \delta_{0},\label{c0small}
\end{align}
then there exists a unique global strong solution $(n,c,u)$ satisfying 
\begin{equation}\label{rThm1.3}
\left\{
\begin{aligned}
&n, c\in \mathcal{C}([0,T];H^{2}(\mathbb{R}^3))\cap L^2(0,T;\dot{H}^{1}\cap \dot{H}^{3}(\mathbb{R}^3)),\\
&u\in \mathcal{C}([0,T];H^{2}(\mathbb{R}^3))\cap L^2(0,T;\dot{H}^{\alpha}\cap \dot{H}^{2+\alpha}(\mathbb{R}^3))
\end{aligned}
\right.
\end{equation}
for any time $T>0$.
\end{theorem}


\vspace{2mm}

 
Finally, under a mild condition of small energy, we have the global existence and optimal time-decay rates of strong solutions to the Cauchy problem~\eqref{eq1-1}-\eqref{eq1-2} in the case $\frac{3}{4}< \alpha<\frac{5}{4}$.

\begin{theorem}\label{thmglobal}
Let $\frac{3}{4}< \alpha<\frac{5}{4}$. Assume \eqref{H1}-\eqref{H2} and further that $\phi$ fulfills $|x|\nabla \phi\in W^{2,\infty}(\mathbb{R}^3)$ and $|x|^{1+\alpha}\nabla \phi\in L^{\infty}(\mathbb{R}^3)$. In addition, suppose that the initial data $(n_{0},c_{0},u_{0})$ fulfills  $n_0, c_0\in L^1(\mathbb{R}^3)$ and $(n_0, c_0, u_0)\in H^{2}(\mathbb{R}^3)$. There exists a constant $\delta_{1}>0$ such that if
\begin{align}
&\|(n_0, c_0,u_{0})\|_{L^2}\leq \delta_{1},\label{smallness2}
\end{align}
then the Cauchy problem \eqref{eq1-1}-\eqref{eq1-2}  admits a unique global strong solution $(n,c,u)$, which satisfies
\begin{equation}\label{uniformglobal}
\begin{aligned}
\sup_{t\in\mathbb{R}_+}\|(n, c, u)\|_{H^{2}}^2+\int_{0}^{\infty}\|(\nabla n,\nabla c, \Lambda^{\alpha}u)\|_{H^{2}}^2\,dt\leq C\|(n_{0}, c_0,u_{0})\|_{H^{2}}^2.
\end{aligned}
\end{equation}
Furthermore, the following time-decay estimates hold{\rm:}
\begin{equation}\label{decay}
\left\{
\begin{aligned}
&\|(n,c)(t)\|_{L^{p}}\leq C(1+t)^{-\frac{3}{2}(1-\frac{1}{p})},\quad\quad 1<p<\infty,\\
&\|(\nabla n,\nabla c)(t)\|_{L^2}\leq C(1+t)^{-\frac{5}{4}},\\
&\|u(t)\|_{L^{p}}\leq C(1+t)^{-\frac{3}{2\alpha}(\frac{1}{2}-\frac{1}{p})},\quad\quad 2<p\leq \frac{6}{3-2\alpha},\\
&\|\Lambda^{\alpha}u(t)\|_{L^2}\leq C(1+t)^{-\frac{1}{2}}.
\end{aligned}
\right.
\end{equation}
\end{theorem}


\begin{remark}
Theorem \ref{thmglobal} provides the first result on the global existence of strong solutions subject to initial data that are of small $L^2$-energy but possibly large oscillations.
In particular, even for the classical well-established case of $\alpha=1$, our findings present novel insights, as the quantity $\|(n_0, c_0, u_0)\|_{H^3}$ 
 in {\rm\cite{duan2010}} was supposed to be sufficiently small.
Moreover, our method allows to obtain uniform bounds with respect to time and establish the decay estimates.
\end{remark}

{\begin{remark}
The time-decay rates in \eqref{decay} are optimal. In fact, the rates of $n_0$ and $c_0$ coincide with the sharp rates of the heat equation under the $L^1$-assumption of initial data, and the rates of $u_0$ are the same as those of the fractional heat equation subject to the initial data in $L^2(\mathbb{R}^3)$. For the sharpness of the rates, one can refer to, e.g. {\rm\cite{BS}}.
\end{remark}}

In what follows, we will make some illustrations for the proof of Theorem \ref{thmglobal}. Our goal is to establish uniform regularity estimates and extend the local solution to the global one. However, due to the fact that the system \eqref{eq1-1} does not possess a dissipative free energy, it is difficult to obtain uniform-in-time evolution without the smallness in $H^2(\mathbb{R}^3)$.  To overcome this, we claim following refined bounds
\begin{align}
\|\nabla c\|_{L^3}^2+\|u\|_{L^3}^2+\int_{0}^{t}\|\nabla u\|_{L^2}^{\frac{4\alpha}{4\alpha-3}}\,d\tau\ll 1. \label{bootstrap}
\end{align}
Such a condition is weaker than higher order smallness assumptions based on the classical linearization argument (e.g., cf. {\rm\cite{duan2010}}). Under \eqref{bootstrap}, we are able to derive some Lyapunov inequalities and study the large-time stability.  First, it is observed that the $L^2$-energy of $c$ is uniform with respect to time due to $f\geq0$, and $n$ satisfies
\begin{equation}\nonumber
\begin{aligned}
&\frac{1}{2}\frac{d}{dt}\|n\|_{L^2}^2+(1-C\|\nabla c\|_{L^3})\|\nabla n\|_{L^2}^2\leq 0.
\end{aligned}
\end{equation}
This leads to the uniform $L^2$-estimate of $n$ in view of \eqref{bootstrap}.  In addition, the uniform $L^2$-estimate of $u$ can be obtained based on  Hardy's inequality and the estimates of $n$ at hand (see Lemma \ref{lemmabasic}). At the $\dot{H}^1$-regularity level, we obtain the following inequality:
\begin{equation}\nonumber
\begin{aligned}
&\frac{1}{2}\frac{d}{dt}\|\nabla c\|_{L^2}^2+(1-C\|u\|_{L^3}^2)\|\Delta c\|_{L^2}^2\leq C\|\nabla n\|_{L^2}^2,
\end{aligned}
\end{equation}
and
\begin{equation}\nonumber
\begin{aligned}
&\frac{1}{2}\frac{d}{dt}\|\nabla n\|_{L^2}^2+(\frac{3}{4}-C\|(\nabla c,u)\|_{L^3}^2)\|\Delta n\|_{L^2}^2\leq C\|\nabla^2 c\|_{L^2}^4 \|\nabla n\|_{L^2}^2,
\end{aligned}
\end{equation}
which allow us to capture the desired $\dot{H}^1$-norms for $c$ and $n$ under the condition \eqref{bootstrap}.
Furthermore, it should be noted that the $L^{\frac{4\alpha}{4\alpha-3}}_{t}(\dot{H}^1)$-norm of $u$ is a scaling-invariant quantity for the generalized incompressible Navier-Stokes equations and plays a key role in controlling the convective term $u\cdot\nabla u$ in $\eqref{eq1-1}$. Indeed, we obtain 
\begin{equation}\nonumber
\begin{aligned}
&\frac{1}{2}\frac{d}{dt}\|\nabla u\|_{L^2}^2+\|\Lambda^{\alpha}\nabla u\|_{L^2}^2\leq C\|\nabla n\|_{L^2}^2+C\|\Lambda^{\alpha}u\|_{L^2}^2+C\|\nabla u\|_{L^2}^{\frac{4\alpha}{4\alpha-3}+2}.
\end{aligned}
\end{equation}
This, together with \eqref{bootstrap} and the $L^2$ dissipation estimates of $n$ and $u$, implies the uniform $\dot{H}^1$-bound of $u$ (see Lemma \ref{lemmahigher}). With these, we also obtain the uniform bounds of the solution at the $\dot{H}^2$-regularity level (cf. Lemma \ref{lemmaH2}). By making use of these estimates and interpolation between $L^2$ energy and higher order norms, we succeed in performing a bootstrap procedure and justify \eqref{bootstrap} uniformly in time. Finally, based on these uniform bounds and Lyapunov-type arguments, we obtain the time-decay estimates of the global solution.

\vspace{2mm}

The rest of this paper is organized as follows. In Section \ref{P}, 
we introduce some notations that are used throughout the paper.  Section \ref{section3} is devoted to the proof of regularity criteria in Theorems \ref{Th1.1} and \ref{Th1.2}.
In Section \ref{section4}, we first show Theorem \ref{Th1.3} about the global existence  for $\alpha \geq \frac{5}{4}$ in Subsection \ref{subsectionglobal1}. Then, under the weaker dissipation ($\frac{3}{4}<\alpha<\frac{5}{4}$) and the mild assumption of small $L^2$-energy, we establish the uniform regularity estimates of global strong solutions with
large oscillations in Subsection \ref{subsectionglobal2} and further obtain the time-decay estimates in Subsection \ref{subsectiondecay}. In Section \ref{sectionextension}, we discuss some possible extensions and problems. Section \ref{section5} is the appendix concerning the proof of the local well-posedness in Theorem  \ref{Thlocal}.

\section{Preliminaries}\label{P}

First, we list some notations that are used frequently throughout the paper.

\paragraph{Notations.} Throughout this paper, $C>0$ denotes some  constant independent of the time, and $C_{T}>0$ is a constant which may depend on the time $T$.  $\mathcal{F}(f)$ and  $\mathcal{F}^{-1}(f)$ stand for the Fourier transform and the inverse transform of the function $f.$  Sometimes we write 
$$
\Lambda^{\sigma}= (-\Delta)^\frac{\sigma}{2}=\mathcal{F}^{-1} \Big( |\xi|^{2\sigma} \mathcal{F}(\cdot ) \Big),\quad\sigma\in\mathbb{R}.
$$
Let 
$W^{s, p}(\mathbb{R}^3)$ with $1\leq p\leq \infty$, $s\in\mathbb{R}$ be the standard fractional exponent Sobolev space on $\mathbb{R}^3$ with equivalent norms
$$
\|f\|_{W^{s, p}} \sim \|({\rm Id }-\Delta)^\frac{s}{2}\|_{L^p} \sim \|f\|_{L^p}+\|\Lambda^s f\|_{L^p}.
$$
 Especially, when $p=2,$ $H^s(\mathbb{R}^3)$ is the Sobolev space of exponent $s$ on $\mathbb{R}^3$ with the standard norm 
$$
\|f\|_{H^s}:=\Big(\int_{\mathbb{R}^3} (1+|\xi|^2)^s|\widehat f (\xi)|^2 \,d\xi \Big)^{\frac{1}{2}}<\infty
$$
and the inner product 
$$\langle f, g\rangle=\int_{\mathbb{R}^3}  (1+|\xi|^2)^s \widehat f (\xi) \overline{\widehat g (\xi)} \,d\xi.$$
Furthermore, we denote $\dot{H}^s(\mathbb{R}^3)$ with $s\in\mathbb{R}$ the homogeneous Sobolev space endowed with the norm
$$
\|f\|_{\dot{H}^s}:=\Big(\int_{\mathbb{R}^3} |\xi|^{2s}|\widehat f (\xi)|^2 \,d\xi \Big)^{\frac{1}{2}}\sim \|\Lambda^{s}f\|_{L^2}.
$$

We need the following generalized Sobolev type estimate.
\begin{lemma}\label{sobolev}
Let $0<\alpha <3$, $1\leq q <p<\infty$ and $\frac{1}{p}+\frac{\alpha}{3}=\frac{1}{q}$. There exists a constant 
$C \geq 0$ such that if $u \in \mathcal{S}'$ is such that 
$\widehat f$ is a function, then 
\begin{equation}\label{eqsobolev}
\|u\|_{L^p}\leq C \|\Lambda^\alpha u\|_{L^q}.
\end{equation}
\end{lemma}

Next lemma is devoted to the Gagliardo–Nirenberg inequality. 
\begin{lemma}[Gagliardo-Nirenberg inequality]\label{lemmaGN}
For $1\leq r \leq \infty,$ $1\leq m <\infty,$ $0<\alpha <3,$ $0\leq \theta \leq 1$ and $\sigma <s$, assume $\Lambda ^\alpha u \in L^m(\mathbb{R}^3) $ and $\Lambda ^s u \in L^r(\mathbb{R}^3).$
If $$\frac{1}{p}-\frac{\sigma}{3}=
\theta\left(\frac{1}{m}-\frac{\alpha}{3} \right)
+(1-\theta)\left(\frac{1}{r}-\frac{s}{3}\right),
$$ then 
\begin{equation}\label{G-GN}
\|\Lambda^\sigma u\|_{L^p} \leq C \|\Lambda^\alpha u\|_{L^m}^\theta \|\Lambda^{s}u\|_{L^{r}}^{1-\theta}.
\end{equation}
\end{lemma}


\begin{proof} By \cite{N} or \cite[Chapter 1]{bahouri1}, one has 
for any $1\leq q,r\leq \infty$, $\sigma<s$, and $\theta\in [0,1]$ and $\frac{1}{p}-\frac{\sigma}{3}=\frac{\theta}{q}+(1-\theta)(\frac{1}{r}-\frac{s}{3})$, if $u\in L^q (\mathbb{R}^3)$ and $\Lambda^s u\in L^r (\mathbb{R}^3)$. Then $\Lambda^{\sigma}u\in L^q(\mathbb{R}^3)$ and
\begin{equation}\label{GN}
\begin{aligned}
&\|\Lambda^{\sigma}u\|_{L^p}\leq C\|u\|_{L^q}^{\theta}\|\Lambda^{s}u\|_{L^{r}}^{1-\theta}.
\end{aligned}
\end{equation}
If let
\begin{equation}\label{beq-1}
\frac{1}{q}+\frac{\alpha}{3}=\frac{1}{m}
\end{equation}
 and 
\begin{equation}\label{beq-2}
\frac{1}{p}=\frac{\theta}{q}+\frac{1-\theta}{r},
\end{equation}
then by Lemma \ref{sobolev} and \eqref{GN}, we obtain
\eqref{G-GN}. 
It can be easily checked that $\frac{1}{p}-\frac{\sigma}{3}=
\theta\left(\frac{1}{m}-\frac{\alpha}{3} \right)
+(1-\theta)\left(\frac{1}{r}-\frac{s}{3}\right)
$ holds due to \eqref{beq-1} and \eqref{beq-2}.
 \end{proof}

\begin{remark} We would like to point out that Lemma \ref{lemmaGN} is valid in all dimensions, meaning that 
for $1\leq r \leq \infty,$ $1\leq m <\infty,$ $0<\alpha <n,$ $0\leq \theta \leq 1$ and $\sigma <s$, if we assume  $\Lambda ^\alpha u \in L^m(\mathbb{R}^n) $ and $\Lambda ^s u \in L^r(\mathbb{R}^n)$ and 
If $$\frac{1}{p}-\frac{\sigma}{n}=
\theta\left(\frac{1}{m}-\frac{\alpha}{n} \right)
+(1-\theta)\left(\frac{1}{r}-\frac{s}{n}\right),
$$ then 
\begin{equation*}
\|\Lambda^\sigma u\|_{L^p(\mathbb{R}^n)} \leq C \|\Lambda^\alpha u\|_{L^m(\mathbb{R}^n)}^\theta \|\Lambda^{s}u\|_{L^{r}(\mathbb{R}^n)}^{1-\theta}.
\end{equation*}
\end{remark}


We also need the Moser-type product laws for fractional Laplacian operators. The first estimate \eqref{product0} is due to Kato-Ponce \cite{KP1} and Kenig-Ponce-Vega \cite{KPV1} and the second one \eqref{product} can be found in \cite[Corollary 2.55]{bahouri1}.
\begin{lemma}\label{lemmaproduct}
The following statement holds{\rm:}
\begin{itemize}
\item Let $s>0$ and $1<p<\infty$. If $u\in L^{p_1}\cap \dot{W}^{s,p_2}(\mathbb{R}^3)$ and $v\in L^{p_1}\cap \dot{W}^{s,p_3}(\mathbb{R}^3)$ with $1<p_1,p_2,p_3,p_3<\infty$ satisfying $\frac{1}{p}=\frac{1}{p_1}+\frac{1}{p_2}=\frac{1}{p_3}+\frac{1}{p_4}$, then we have $uv\in \dot{W}^{s,p}(\mathbb{R}^3)$ and
\begin{align}
\|\Lambda^{s}(uv)\|_{L^p}\leq C\|u\|_{L^{p_1}} \|\Lambda^{s}
v\|_{L^{p_2}}+C\|\Lambda^{s}u\|_{L^{p_3}}\|v\|_{L^{p_4}}.\label{product0}
\end{align}

\item  Let $s_{1},s_{2}$ satisfy $-\frac{3}{2}< s_{1},s_{2}<\frac{3}{2}$ and $s_{1}+s_{2}>0$. If $u\in \dot{H}^{s_1}(\mathbb{R}^3)$ and $v\in \dot{H}^{s_2}(\mathbb{R}^3)$, then we have $uv\in \dot{H}^{s_{1}+s_{2}-\frac{3}{2}}(\mathbb{R}^3)$ and
\begin{align}
\|\Lambda^{s_1+s_2-\frac{3}{2}}(uv)\|_{L^2}\leq C\|\Lambda^{s_1}u\|_{L^2}\|\Lambda^{s_2} v\|_{L^2}.\label{product}
\end{align}
\end{itemize}
\end{lemma}

Finally, we recall Hardy's inequality for fractional norms (see \cite[Page 91]{bahouri1}).

\begin{lemma}\label{lemmahardy}
Let $0\leq s<\frac{3}{2}$. For any $v\in \dot{H}^{s}(\mathbb{R}^{3})$, it holds that 
\begin{equation}\label{hardy}
\begin{aligned}
&\Big(\int_{\mathbb{R}^{3}}\frac{|v(x)|^2}{|x|^{2s}}\,dx\Big)^{\frac{1}{2}}\leq C\|\Lambda^{s}v\|_{L^2}.
\end{aligned}
\end{equation}
\end{lemma}

\section{Regularity criteria}\label{section3}

In this section, we aim at investigating the mechanism of possible finite time blow-up for the generalized chemotaxis-Navier-Stokes model \eqref{eq1-1}, which enables us to study the global existence if the corresponding norms of these criteria can be controlled.

\subsection{Proof of Theorem \ref{Th1.1}}\label{subsection31}


We will give the proof of Theorem \ref{Th1.1}.  We prove the theorem by contradictory arguments. Assume that the maximal time $T_*$ is finite, and \eqref{eq1-3} is not true, i.e., 
\begin{equation}
B_{1}:=\lim_{T\rightarrow T_*}\int_0^{T}\left(\|n\|_{L^{q_1}}^{p_1}+\|u\|_{L^{q_2}}^{p_2} \right)\,dt <+\infty\label{blow1}
\end{equation}
for any $(p_i, q_i) (i=1, 2)$ satisfying  \eqref{piqi}. Then, once the $H^2$-norm of the solution $(n, c, u)$ to the Cauchy problem \eqref{eq1-1}-\eqref{eq1-2} is bounded when $t$ is close to $T_*$, one can show that the assumption $T_*<+\infty$ will lead to a contradiction with the maximality of $T_*$ in accordance with the local well-posedness result (Theorem~\ref{Thlocal}). 

In the following, we establish the key estimates of solutions. The proof is divided into three steps.




\begin{itemize}
\item \textbf{Step 1: $L^2$-estimates of $n$ and $H^1$-estimates of $c,u$.}
\end{itemize}

In this step, our goal is to get
\begin{equation}\label{eq3-0}
\sup_{t\in(0,T]}(\|n\|_{L^2}^2+\|(c,u)\|_{H^1}^2)+\int_{0}^{T}(\|\nabla n\|_{L^2}^2+\|\nabla c\|_{H^1}^2+\|\Lambda^{\alpha}u\|_{H^1}^2)\,dt\leq C(B_1,T_*)
\end{equation}
for any $0<T<T_*$ and some constant $C(B_1,T_*)>0$.

To begin with, we give some uniform estimates. It follows from  the classical maximum principle that
\begin{equation}\label{cinfty}
\begin{aligned}
&n(t,x)\geq 0,\quad\quad 0\leq c(t,x)\leq \|c_0\|_{L^{\infty}},\quad (t,x)\in [0,T]\times \mathbb{R}^3.
\end{aligned}
\end{equation}
In addition, by taking the inner product of $\eqref{eq1-1}_2$ by $c$ and using $f(c)\geq0$, we arrive at
\begin{equation}\label{ccL2}
\begin{aligned}
\sup_{t\in(0,T]}\|c\|_{L^2}^2+\int_{0}^{T}\|\nabla c\|_{L^2}^2\,dt\leq \|c_0\|_{L^2}^2.
\end{aligned}
\end{equation}

In the following, we divide the proof of \eqref{eq3-0} into the cases $\alpha\geq 1$ and $\frac{3}{4}<\alpha<1$ several times.

 \begin{itemize}

\item {\emph{ Case}} 1:   $\alpha\geq 1$
\end{itemize}

Now, we perform the $L^2$-estimate of $n$ and the $H^1$-estimates of $c,u$. First, we multiply $\eqref{eq1-1}_1$ by $n$, integrate the resulting equation by parts, and then derive 
\begin{equation}\nonumber
\begin{aligned}
	\frac{1}{2}\frac{d}{dt}\|n\|_{L^2 }^2+\|\nabla n\|_{L^2 }^2 
	&=
	\int_{\mathbb{R}^3}\chi(c) n\nabla c\cdot \nabla n\,dx\\
	&\leq   \sup_{0\leq c\leq \|c_0\|_{L^\infty }} |\chi (c)| \|\nabla n\|_{L^2 } \|n\|_{L^{q_1} }\|\nabla c\|_{L^\frac{2q_1}{q_1-2}}\\
 &\leq C \|\nabla n\|_{L^2 } \|n\|_{L^{q_1} }  \|\nabla c\|_{L^{2} }^\frac{q_1-3}{q_1} \|\nabla^2c\|_{L^2 }^\frac{3}{q_1}\\
		& \leq  \frac{1}{2} \|\nabla n\|_{L^2 }^2+\frac{1}{4}\|\Delta c\|_{L^2}^2+C\|n\|_{L^{q_1} }^{\frac{2q_1}{q_1-3}}\|\nabla c\|_{L^2 }^2,
\end{aligned}
	\end{equation}
	where we have used \eqref{cinfty}, $3<q_1<\infty$, the Gagliardo-Nirenberg inequality \eqref{G-GN} and  the fact that
\begin{align}
\|\Delta g\|_{L^2}\leq \|\nabla^2g\|_{L^2}\leq C\|\Delta g\|_{L^2},\quad \forall g\in\dot{H}^2.\label{Deltasim}
\end{align}
Due to $\frac{2q_1}{q_1-3}\leq p_1$, it thus follows that 
	\begin{equation}\label{eq3-2}
	\frac{d}{dt}\|n\| _{L^2 }^2+\|\nabla n\|_{L^2 }^2	\leq \frac{1}{2}\|\Delta c\|_{L^2}^2+C(1+\|n\|_{L^{q_1}}^{p_1})\|\nabla c\| _{L^2 }^2.
\end{equation}

Furthermore, we multiply both sides of  $\eqref{eq1-1}_2$ by $-\Delta c$ and integrate the resulting equation to obtain  
\begin{equation}\label{nablac}
	\frac{1}{2} \frac{d}{dt}\|\nabla c\| _{L^2 }^2+\|\Delta c\| _{L^2 }^2
	=\int_{\mathbb{R}^3} (u\cdot\nabla c)\cdot\Delta c\, dx+\int_{\mathbb{R}^3} n \Delta c  f(c)\,dx.
\end{equation}
Due to $2\alpha-1\geq 1$, we have $\frac{2q_2}{q_2-3}\leq p_2$, so it follows from \eqref{G-GN} and \eqref{Deltasim} that
\begin{equation}\nonumber
\begin{aligned}
\left|\int_{\mathbb{R}^3} (u\cdot\nabla c)\cdot\Delta c\, dx\right|&\leq \|u\|_{L^{q_2}}\|\nabla c\|_{L^\frac{2q_2}{q_2-2}} \|\Delta c\|_{L^2} \\
&\leq C \|u\|_{L^{q_2}} \|\nabla c\|_{L^2}^{\frac{q_2-3}{q_2}}\|\nabla^2c\|_{L^2}^{\frac{3}{q_2}}\|\Delta c\|_{L^2}\\
&\leq \frac{1}{4}\|\Delta c\|_{L^2}^2+C\|u\|_{L^{q_2}}^{\frac{2q_2}{q_2-3}} \|\nabla c\|_{L^2}^2\\
&\leq \frac{1}{4}\|\Delta c\|_{L^2}^2+C(1+\|u\|_{L^{q_2}}^{p_2}) \|\nabla c\|_{L^2}^2.
\end{aligned}
\end{equation}
And by \eqref{cinfty}, it is easy to verify that
\begin{eqnarray*}
\begin{aligned}
\left|\int_{\mathbb{R}^3} n \Delta c  f(c)\,dx\right|&\leq \frac{1}{4}\|\Delta c\|_{L^2}^2+\sup_{0\leq s\leq \|c_0\|_{L^\infty}}|f(s)|^2
	\|n\| _{L^2 }^2.
\end{aligned}
\end{eqnarray*}
Hence, we have
\begin{equation}\label{eq3-22222}
\begin{aligned}
& \frac{d}{dt}\|\nabla c\| _{L^2 }^2+\|\Delta c\| _{L^2 }^2\leq C(1+\|u\|_{L^{q_2}}^{p_2}) \|\nabla c\|_{L^2}^2+C\|n\| _{L^2 }^2.
\end{aligned}
\end{equation}

We are going to analyze the velocity $u$. For the third equation in \eqref{eq1-1}, one has
\begin{equation}\label{eq3-4}
\begin{aligned}
\frac{1}{2}\frac{d}{dt}\|u\| _{L^2 }^2+\|\Lambda^{\alpha} u\| _{L^2 }^2
=-\int_{\mathbb{R}^3} u\cdot n\nabla \phi 
\leq \|u\| _{L^2 }^2+C \|\nabla \phi\|_{L^\infty}^2\|n\| _{L^2 }^2.
\end{aligned}
	\end{equation}
To deal with the $\dot{H}^1$-estimates of $u$, we multiply both sides of the third equation of \eqref{eq1-1} by $-\Delta u$ and integrate by parts, and then have
\begin{equation}\label{eq3-6}
	\begin{aligned}
	\frac{1}{2} \frac{d}{dt} \|\nabla u\| _{L^2 }^2	+\|\Lambda ^{\alpha+1} u\| _{L^2 }^2
	=  \int_{\mathbb{R}^3} \Delta u \cdot n \nabla \phi\,dx+\int_{\mathbb{R}^3} \Delta u \cdot (u \cdot \nabla) u\,dx.
	\end{aligned}
\end{equation}
For the terms on the right-hand side of \eqref{eq3-6}, we first deduce from  the Gagliardo-Nirenberg inequality \eqref{G-GN} and $\frac{2q_1}{q_1-3}\leq p_1$ that
\begin{equation}\label{eq3-7}
	\begin{aligned}
\left|\int_{\mathbb{R}^3} \Delta u \cdot n \nabla \phi\,dx \right|&\leq  \|\nabla \phi\|_{L^\infty}\|\Delta u\| _{L^2 }\|n\| _{L^2 }\\
&\leq  C \|\nabla u\| _{L^2 }^{1-\frac{1}{\alpha}}
\|\Lambda^{\alpha+1} u\| _{L^2 }^\frac{1}{\alpha}\|n\| _{L^2 }\\
&\leq  \frac{1}{4} \|\Lambda^{\alpha+1} u\| _{L^2 }^2
+C\|\nabla u\| _{L^2 }^2+C\|n\| _{L^2 }^2.
	\end{aligned}
\end{equation}
Similarly, it holds by \eqref{G-GN} that
\begin{equation}\label{eq3-8}
\begin{aligned}
\left|\int_{\mathbb{R}^3} \Delta u \cdot (u \cdot \nabla) u	\,dx\right|&=\left| \sum_{i,j,k=1}^{3}\int_{\mathbb{R}^3}  u_k(\partial_{ik}u_j\cdot\partial_i u_j  
+\partial_k u_j \cdot \partial_{ii} u_j)\,dx
\right|\\
&\leq C\|u\|_{L^{q_2}}\|\nabla u\|_{L^{r_1}}\|\Delta u\|_{L^{r_2}}\\
&\leq   C\|u\|_{L^{q_2}}  \|\nabla u\| _{L^2 }^\theta  \|\Lambda^{\alpha +1}u\| _{L^2 }^{1-\theta}
\|\nabla u\| _{L^2 }^\theta  \|\Lambda^{\alpha +1}u\| _{L^2 }^{1-\theta}\\
& \leq  \frac{1}{4} \|\Lambda^{\alpha +1}u\| _{L^2 }^2+
C\|u\|_{L^{q_2}}^\frac{1}{\theta}\|\nabla u\| _{L^2 }^2,
\end{aligned}	
\end{equation}
where we used  $\alpha\geq1$, and the constants $r_1, \, r_2$ and $\theta$ satisfy 
\begin{eqnarray*}
\begin{cases}
\frac{1}{r_1}+\frac{1}{r_2}+\frac{1}{q_2}=1,&\\
\frac{1}{r_1}-\frac{1}{3}=\theta\left(\frac{1}{2}-\frac{1}{3}\right)+
(1-\theta)\left(\frac{1}{2}-\frac{\alpha+1}{3}\right)&\\
\frac{1}{r_2}-\frac{2}{3}=\theta\left(\frac{1}{2}-\frac{1}{3}\right)+
(1-\theta)\left(\frac{1}{2}-\frac{\alpha+1}{3}\right).&	
\end{cases}	
\end{eqnarray*}
which can be rewritten  as 
$$
r_1=\frac{6q_2}{2q_2-3},\,\,\quad  r_2=\frac{6q_2}{4q_2-3},\,\, \quad \frac{1}{p_2}\leq \theta=1-\frac{1}{2\alpha}\left(1+\frac{3}{q_2}\right). 
$$
In order to promise $1<r_1, \, r_2<\infty$ and $0< \theta < 1$, one requires $\min\{\frac{3}{2},\frac{3}{2\alpha-1}\}<q_2<\infty$.
Substituting  \eqref{eq3-7} and \eqref{eq3-8} into \eqref{eq3-6}, we conclude that 
\begin{equation}\label{nablaucase1}
	\begin{aligned}
	\frac{d}{dt}\|\nabla u\| _{L^2 }^2	+\|\Lambda ^{\alpha+1} u\| _{L^2 }^2
\leq  C \|n\|_{L^2}^2+\left(1+\|u\|_{L^{q_2}}^{p_2}\right)\|\nabla u\| _{L^2 }^2.
		\end{aligned}
\end{equation}
The combination of \eqref{eq3-2}, \eqref{eq3-22222},  \eqref{eq3-4} and \eqref{nablaucase1} yields
\begin{equation}\nonumber
	\begin{aligned}
	&\frac{d}{dt}(\|(n,\nabla c)\|_{L^2}^2+\| u\| _{H^1 }^2)	+\frac{1}{2}\|(\nabla n,\Delta c)\|_{L^2}^2+ \frac{1}{2}\|\Lambda ^{\alpha} u\| _{H^1}^2\\
&\leq C(1+\|n\|_{L^{q_1}}^{p_1}+\|u\|_{L^{q_2}}^{p_2})(\|(n,\nabla c)\|_{L^2}^2+\| u\| _{H^1 }^2),
		\end{aligned}
\end{equation}
which, together with \eqref{blow1}, \eqref{Deltasim} and Gr\"onwall's lemma, yields
\begin{equation}\nonumber
	\begin{aligned}
 &\sup_{t\in(0,T]}(\|(n,\nabla c)\|_{L^2}^2+\| u\| _{H^1 }^2)+\int_{0}^{T}(\|(\nabla n,\nabla^2 c)\|_{L^2}^2+ \|\Lambda ^{\alpha} u\| _{H^1}^2)\,dt\\
 &\quad \leq Ce^{CB_1+C T}(\|(n_0,\nabla c_0)\|_{L^2}^2+\| u_0\|_{H^1}^2)
		\end{aligned}
\end{equation}
for any $0<T<T_*$. This together with \eqref{ccL2} yields \eqref{eq3-0} in the case $\alpha\geq 1$.

\begin{itemize}

\item {\emph{ Case}} 2:   $\frac{3}{4}< \alpha< 1$
\end{itemize}

In this case, we are able to obtain the inequality \eqref{eq3-2} for $n$ as in Case 1. Regarding the estimates of $c$, the second term on the right-hand side of \eqref{nablac} can be handled in the same way in Case 1. To deal with the remainder on the right-hand side of \eqref{nablac}, one has
\begin{equation}\nonumber
\begin{aligned}
\int_{\mathbb{R}^3} (u\cdot\nabla c)\cdot\Delta c\, dx&=-\int_{\mathbb{R}^3} (\nabla c \cdot \nabla u)\cdot \nabla c\,dx-\int_{\mathbb{R}^3}(u\cdot \nabla^2 c)\cdot \nabla c\,dx \\
&=\int_{\mathbb{R}^3} \nabla^2 c: \nabla u c\,dx,
\end{aligned}
\end{equation}
derived from $\nabla\cdot u=0$ and integration by parts, so it holds by \eqref{G-GN}, \eqref{cinfty} and \eqref{Deltasim} that
\begin{equation}\nonumber
\begin{aligned}
\left|\int_{\mathbb{R}^3} (u\cdot\nabla c)\cdot\Delta c\, dx\right|&\leq \|c\|_{L^{\infty}} \|\nabla u\|_{L^2}\|\nabla^2c\|_{L^2}\\
&\leq C\|c_0\|_{L^{\infty}} \|u\|_{L^2}^{\frac{\alpha}{\alpha+1}}\|\nabla\Lambda^\alpha u\|_{L^2}^{\frac{1}{\alpha+1}} \|\Delta c\|_{L^2}\\
&\leq \frac{1}{4} \|\Delta c\|_{L^2}^2+\frac{1}{4}\|\nabla \Lambda^{\alpha}u\|_{L^2}^2+C \|u\|_{L^2}^2.
\end{aligned}
\end{equation}
Therefore, we gain
\begin{equation}\label{eq3-222220}
\begin{aligned}
& \frac{d}{dt}\|\nabla c\| _{L^2 }^2+\|\Delta c\| _{L^2 }^2\leq \frac{1}{2}\|\nabla \Lambda^{\alpha}u\|_{L^2}^2+C\|(n,u)\| _{L^2 }^2.
\end{aligned}
\end{equation}

 Furthermore, we are going to establish the $\dot{H}^1$-estimates of $u$. In fact, multiplying both sides of the third equation of \eqref{eq1-1} by $-\Delta u$ and integrating by parts, one has 
\begin{equation}\label{eq22-5}
	\begin{aligned}
	\frac{1}{2} \frac{d}{dt} \|\nabla u\| _{L^2 }^2	+\|\nabla \Lambda ^{\alpha} u\| _{L^2 }^2
	=  \int_{\mathbb{R}^3} \Delta u \cdot n \nabla \phi\,dx+\int_{\mathbb{R}^3} \Delta u \cdot (u \cdot \nabla) u \,dx.
	\end{aligned}
\end{equation}
The first term on the right-hand side of \eqref{eq22-5}  is analyzed by
\begin{equation}\label{case2uphi}
	\begin{aligned}
\left|\int_{\mathbb{R}^3} \Delta u \cdot n \nabla \phi\,dx \right|&=\left|\int_{\mathbb{R}^3} \nabla u \cdot \nabla (n \nabla \phi)\,dx \right|\\
&\leq  \|\nabla u\|_{L^2}\|n\| _{L^2 }\|\nabla^2 \phi\|_{L^\infty}+\|\nabla u\|_{L^2}\|\nabla n\| _{L^2 }\|\nabla \phi\|_{L^\infty}\\
&\leq  C\|\nabla u\| _{L^2 }^2+C\|n\| _{L^2 }^2+\frac{1}{4}\|\nabla  n\| _{L^2 }^2.
	\end{aligned}
\end{equation}
As $\Lambda^{\sigma}$ is self-adjoint for any $\sigma>0$, one deduces from  $\nabla\cdot u=0$ and \eqref{G-GN} that
\begin{equation}\nonumber
\begin{aligned}
\Big| \int_{\mathbb{R}^3} \Delta u \cdot (u \cdot \nabla) u \,dx\Big|&=\Big| \int_{\mathbb{R}^3} \Lambda^{\alpha-1}\Delta u \cdot \Lambda^{1-\alpha} \div (u \otimes u) \,dx\Big|\\
&\leq C\|\nabla \Lambda^{\alpha}u\|_{L^2} \|\Lambda^{2-\alpha}(u\otimes u)\|_{L^2}.
\end{aligned}	
\end{equation}
Note that the product law \eqref{product0} and the Gagliardo-Nirenberg inequality \eqref{G-GN} implies that
\begin{equation}\nonumber
\begin{aligned}
\|\Lambda^{2-\alpha}(u\otimes u)\|_{L^2}\leq C\|u\|_{L^{q_2}} \|\Lambda^{2-\alpha}u\|_{L^{\frac{2q_2}{q_2-2}}}\leq C\|u\|_{L^{q_2}} \|\nabla u\|_{L^2}^{1-\theta} \|\nabla\Lambda^{\alpha}u\|_{L^2}^{\theta},
\end{aligned}	
\end{equation}
where $\theta=\frac{1}{\alpha}(1+\frac{3}{q_2})-1$ fulfills $$
\frac{q_2-2}{2q_2}-\frac{2-\alpha}{3}=\left(\frac{1}{2}-\frac{1}{3}\right)(1-\theta)+\left(\frac{1}{2}-\frac{1+\alpha}{3}\right)\theta.
$$
Since $\frac{2}{1-\theta}\leq p_2$ due to \eqref{piqi}, we obtain
\begin{equation}\nonumber
\begin{aligned}
\Big| \int_{\mathbb{R}^3} \Delta u \cdot (u \cdot \nabla) u \,dx\Big| &\leq \frac{1}{2}\|\nabla\Lambda^{\alpha}u\|_{L^2}^2+C\|u\|_{L^{q_2}}^{\frac{2}{1-\theta}}\|\nabla u\|_{L^2}^2\\
&\leq  \frac{1}{2}\|\nabla\Lambda^{\alpha}u\|_{L^2}^2+C(1+\|u\|_{L^{q_2}}^{p_2})\|\nabla u\|_{L^2}^2.
\end{aligned}	
\end{equation}
It thus follows that
\begin{equation}\label{ucase2}
\begin{aligned}
\frac{d}{dt} \|\nabla u\| _{L^2 }^2	+\|\nabla \Lambda ^{\alpha} u\| _{L^2 }^2&\leq C(1+\|u\|_{L^{q_2}}^{p_2})\|\nabla u\|_{L^2}^2+C\|n\|_{L^2}^2+\frac{1}{2}\|\nabla  n\| _{L^2 }^2.
\end{aligned}
\end{equation}
Collecting the estimates  \eqref{eq3-2}, \eqref{eq3-7} and \eqref{eq3-222220} and \eqref{ucase2} and thence applying \eqref{blow1} and Gr\"onwall's lemma, we arrive at
\begin{equation*}
	\begin{aligned}
 &\sup_{t\in(0,T]}(\|(n,\nabla c)\|_{L^2}^2+\| u\| _{H^1 }^2)+\int_{0}^{T}(\|(\nabla n,\nabla^2 c)\|_{L^2}^2+ \|\Lambda ^{\alpha} u\| _{H^1}^2)\,dt\\
 &\quad\leq e^{CT+CB_1}(\|(n_0,\nabla c_0)\|_{L^2}^2+\|u_0\|_{H^1}^2).
	\end{aligned}
\end{equation*}
Together with \eqref{ccL2}, we thus establish \eqref{eq3-0} for all $\alpha>\frac{3}{4}$.

\vspace{2mm}

\begin{itemize}
\item \textbf{Step 2: $H^1$-estimates of $n$ and $H^2$-estimates of $c,u$.}
\end{itemize}

In this step, we claim that for any $0<T<T_*$, there exists some constant $C(B_1,T_*)$ such that
\begin{equation}\label{eq3-10-1}
\sup_{t\in(0,T]}\|(\nabla n,\nabla^2 c,\nabla^2 u)\|_{L^2}^2+\int_{0}^{T}\|(\nabla^2 n,\nabla^3 c,\Lambda^{\alpha+2}u)\|_{L^2}^2\,dt\leq C(B_1,T_*).
	\end{equation}

Despite the $H^1$-estimate of $c$ obtained in Step 1,  it is not enough to 
control the chemotaxis term $\nabla \cdot (\chi(c) n \nabla c)$ in $\eqref{eq1-1}_1$. To this end, we first establish the  $H^2$-estimate of $c$. Applying $\Delta$ to $ \eqref{eq1-1}_2$ yields
\begin{equation}
\partial_t \Delta c-\Delta \Delta c=-\Delta(u \cdot \nabla c)-\Delta(nf(c)).\label{Deltac}
\end{equation}
Multiplying \eqref{Deltac} by $\Delta c$, we obtain 
\begin{equation}\label{eq3-11}
\begin{aligned}
\frac{1}{2}\frac{d}{dt}\|\Delta c\| _{L^2 }^2	+\|\nabla\Delta c\| _{L^2 }^2
=\int_{\mathbb{R}^3} \nabla \Delta c \cdot \nabla (u \cdot \nabla c)\,dx
+\int_{\mathbb{R}^3}\nabla \Delta c \cdot \nabla(n f(c))\,dx.
\end{aligned}
\end{equation}
In view of \eqref{G-GN} and Sobolev's inequality, the terms on the right hand side of \eqref{eq3-11} can be analyzed as 
\begin{equation}\label{eq3-12}
\begin{aligned}
&\left|\int_{\mathbb{R}^3} \nabla \Delta c \cdot \nabla (u \cdot \nabla c)\,dx \right|\\
& \quad\leq  \|\nabla\Delta c\|_{L^2 }\left(
\|\nabla u\|_{L^3} \|\nabla c\|_{L^6}+
\|u\|_{L^\infty}\|\nabla^2 c\| _{L^2 }
\right)\\
&\quad \leq C \|\nabla \Delta c\| _{L^2 }
\left(
\|\Lambda^{\alpha +1} u\| _{L^2 }^\frac{3}{2(1+\alpha)}
\|u\| _{L^2 }^{1-\frac{3}{2(1+\alpha)}}
\|\nabla^2 c\| _{L^2 }+\|u\|_{H^{\alpha+1}}\|\Delta c\| _{L^2 }\right)\\
&\quad\leq \frac{1}{4}\|\nabla\Delta c\| _{L^2 }^2+C
\|u\|_{H^{\alpha+1}}^2	\|\Delta c\| _{L^2 }^2,
\end{aligned}
\end{equation}
and
\begin{equation}\label{eq3-13}
\begin{aligned}
&\left|\int_{\mathbb{R}^3}\nabla \Delta c \cdot \nabla(n f(c))\,dx\right|\\
&\quad\leq \|\nabla\Delta c\|_{L^2}\Big(\sup_{0\leq s\leq \|c_0\|_{L^\infty}}|f(s)| \|\nabla n\|_{L^2}+\sup_{0\leq s\leq \|c_0\|_{L^\infty}}|f'(s)| \|n\|_{L^3} \|\nabla c\|_{L^6}\Big)\\
&\quad\leq C\|\nabla\Delta c\|_{L^2}(\|\nabla n\|_{L^2}+\|n\|_{L^2}^{\frac{1}{2}}\|\nabla n\|_{L^2}^{\frac{1}{2}} \|\Delta c\|_{L^2})\\
&\quad\leq   \frac{1}{4}\|\nabla\Delta c\| _{L^2 }^2+C \|\nabla n \| _{L^2 }^2+C(\|n\|_{L^2}^2+\|\nabla n\|_{L^2}^2)\|\Delta c\|_{L^2}^2.
\end{aligned}
\end{equation}
Substituting \eqref{eq3-12} and \eqref{eq3-13} into \eqref{eq3-11}, one has 
\begin{equation*}
\begin{aligned}
\frac{d}{dt}\|\Delta c \| _{L^2 }^2+\|\nabla \Delta c \| _{L^2 }^2
\leq  C(\|n\|_{L^2}^2+\|\nabla n\|_{L^2}^2+\|u\|_{H^{\alpha+1}}^2)\|\Delta c\|_{L^2}^2+C\|\nabla n\| _{L^2 }^2,
\end{aligned}
\end{equation*}
which, together with \eqref{eq3-0} and Gr\"{o}nwall's inequality, implies that, for any $t\in(0,T)$ with $0<T<T_*$, 
\begin{equation}\label{eq3-14}
\begin{aligned}
&\sup_{t\in(0,T]} \|\nabla^2 c \| _{L^2 }^2+
\int_0^{T}\|\nabla^3 c \| _{L^2 }^2 \,dt\\
&~~\leq C {\rm exp} \left\{\int_0^{T}(\|\nabla n\|_{L^2}^2+\|u\|_{H^{\alpha+1}}^2)\,dt+T\sup_{t\in(0,T)}\|n\|_{L^2}^2\right\}\\
&~~\quad\quad\quad\cdot\left(
\|\nabla^2 c_0\| _{L^2 }^2+\int_0^{T} \|\nabla n\| _{L^2 }^2 \,dt
\right)\leq  C(B_1,T).
\end{aligned}
	\end{equation}

Next, we are ready to bound $\nabla n$. Taking the $L^2$-inner product of $\eqref{eq1-1}_1$ with $\Delta n$ yields
\begin{equation}\label{sssssss}
\begin{aligned}
&\frac{1}{2}\frac{d}{dt}\|\nabla n\|_{L^2}^2+\|\Delta n\|_{L^2}^2\leq \frac{1}{2}\|\Delta n\|_{L^2}^2+\|u\cdot\nabla n\|_{L^2}^2+\|\nabla\cdot(\chi(c)n\nabla c)\|_{L^2}^2.
\end{aligned}
\end{equation}
Since $\alpha >\frac{1}{2}$, we can use \eqref{G-GN} and the Young inequality to obtain 
\begin{equation}\nonumber
\begin{aligned}
\|u\cdot\nabla n\|_{L^2}^2\leq \|u\|_{L^{\infty}}^2 \|\nabla n\|_{L^2}^2\leq C(\|u\|_{L^2}^2+\|\Lambda^{1+\alpha}u\|_{L^2}^2)\|\nabla n\|_{L^2}^2,
\end{aligned}
\end{equation}
and
the last term in \eqref{sssssss} can be controlled by
\begin{equation}\nonumber
\begin{aligned}
\|\nabla\cdot(\chi(c)n\nabla c)\|_{L^2}^2&\leq C\sup_{0\leq s\leq \|c_0\|_{L^{\infty}}}|\chi'(s)| \|\nabla c\|_{L^6}^4\|n\|_{L^6}^2\\
&\quad+C\sup_{0\leq s\leq \|c_0\|_{L^{\infty}}}|\chi(s)|(\|\nabla n\|_{L^2}^2\|\nabla c\|_{L^\infty}^2+\|n\|_{L^{6}}^2\|\nabla^2 c\|_{L^3}^2)\\
&\leq C(\|\nabla^2 c\|_{L^2}^4+\|\nabla^2c\|_{H^1}^2)\|\nabla n\|_{L^2}^2.
\end{aligned}
\end{equation}
Hence, using Gr\"onwall's lemma and the estimates \eqref{eq3-0} and \eqref{eq3-14} at hand, for any $0<T<T_*$, we obtain 
\begin{equation}\label{eq3-17}
\begin{aligned}
&\sup_{t\in(0,T]}\|\nabla n\| _{L^2 }^2+
\int_0^{T}\|\nabla^2 n\| _{L^2 }^2 \,dt\\
&\leq C\|\nabla n_0\| _{L^2 }^2 
   \cdot {\rm exp} \left\{
\int_0^{T}
\left(\|\nabla^2 c\| _{L^2 }^4
+\|\nabla c\|_{H^2}^2+\|u\|_{L^2}^2+\|\Lambda^{1+\alpha}u\|_{L^2}^2\right)\,dt
\right\}\leq C(B_1,T).
\end{aligned}
\end{equation}

Regarding the estimates of $u$, applying $\Delta$ to both side of $\eqref{eq1-1}_3$ and then taking the inner product of the resulting equation by $\Delta u,$ we have 
\begin{equation}\label{eq3-18}
\begin{aligned}
\frac{1}{2}\frac{d}{dt}&\|\Delta u\| _{L^2 }^2
+\|\Lambda^{\alpha }\Delta u\| _{L^2 }^2\\
&= \int_{\mathbb{R}^3}\nabla \Delta u \cdot \nabla ((u \cdot \nabla u)u)\,dx+
\int_{\mathbb{R}^3}\nabla \Delta u \cdot
\nabla (n\nabla \phi)\,dx.
\end{aligned}
\end{equation}
To proceed with the estimate, we deal with the two terms at the right hand side of \eqref{eq3-18} respectively. 
For any $\alpha>\frac{3}{4}$, we note that
\begin{equation}\nonumber
\begin{aligned}
 \int_{\mathbb{R}^3}\nabla \Delta u \cdot \nabla ((u \cdot \nabla) u)\,dx&=\sum_{i,j,k,m=1}^3\int_{\mathbb{R}^3}\partial_{iij}u^m \partial_{j}(u^k\partial_{k}u^m)\,dx\\
 &=-\sum_{i,j,k,m=1}^3\int_{\mathbb{R}^3}\partial_{ii}u^m \partial_{jj}u^k \partial_{k}u^m\,dx-2\sum_{i,j,k,m=1}^3 \int_{\mathbb{R}^3}\partial_{ii}u^m \partial_{j}u^k\partial_{jk} u^m\,dx,
	\end{aligned}
	\end{equation}
so it follows from the Gagliardo-Nirenberg inequality \eqref{G-GN} that
\begin{equation}\label{eq3-20}
\begin{aligned}
	 \left|\int_{\mathbb{R}^3}\nabla \Delta u \cdot \nabla ((u \cdot \nabla) u)\,dx\right|
	&\leq  3 \|\nabla u\| _{L^2 } \|\nabla^2 u\|_{L^4}^2\\
	&\leq  C\|\nabla u\| _{L^2 }
	\|\nabla^2 u\| _{L^2 }^{2-\frac{3}{2\alpha}}
	\|\Lambda^{2+\alpha} u\| _{L^2 }^{\frac{3}{2\alpha}}\\
	&\leq  \frac{1}{4} \|\Lambda^{2+\alpha} u\| _{L^2 }^2
	+ C \|\nabla u\| _{L^2 }^{\frac{4\alpha}{4\alpha-3}}\|\nabla^2 u\| _{L^2 }^2 \\
	&\leq \frac{1}{4} \|\Lambda^{2+\alpha} u\| _{L^2 }^2
	+ C \|\nabla^2 u\| _{L^2 }^2
	\|\Lambda^{\alpha+1} u\| _{L^2 }^{\frac{4\alpha}{(4\alpha-3)(\alpha+1)}}\|u\| _{L^2 }^\frac{4\alpha^2}{(4\alpha-3)(\alpha+1)}\\
	&\leq \frac{1}{4} \|\Lambda^{2+\alpha} u\| _{L^2 }^2
	+ C\left(1+\|\Lambda^{\alpha+1} u\| _{L^2 }^2+\| u\| _{L^2 }^2\right) \|\Delta u\| _{L^2 }^2.
	\end{aligned}
	\end{equation}
For the second term at the right hand side of \eqref{eq3-18}, in the case $\alpha\geq1$, one gets from the Gagliardo-Nirenberg inequality \eqref{G-GN} that
\begin{equation}\label{eq3-21}
\begin{aligned}
\left| \int_{\mathbb{R}^3}\nabla \Delta u \cdot \nabla (n\nabla \phi)\,dx\right|
 &\leq  \|\nabla^3 u\| _{L^2 }\|\nabla( n \nabla \phi)\| _{L^2 }^2\\
&\leq  C\|\Lambda^{\alpha+2} u\| _{L^2 }^\frac{6}{2+\alpha}
\|u\| _{L^2 }^{2-\frac{6}{2+\alpha}}
+
C\|\nabla\phi\|_{W^{1,\infty}}^2\|n\|_{H^1}^2
\\
&\leq \frac{1}{4} \|\Lambda^{\alpha+2} u\| _{L^2 }^2+
C\left(\|u\| _{L^2 }^2
+\| n\|_{H^1}^2\right).
\end{aligned}
\end{equation}
As for the case $\alpha<1$, one also has
\begin{equation}\label{eq3-210p}
\begin{aligned}
\left| \int_{\mathbb{R}^3}\nabla \Delta u \cdot \nabla (n\nabla \phi)\,dx\right|
 &\leq \|\Lambda^{\alpha-1}\nabla \Delta u\|_{L^2}\|\Lambda^{1-\alpha}\nabla (n\nabla \phi)\|_{L^2}\\
 &\leq C\|\Lambda^{\alpha}\Delta u\|_{L^2} \| \Lambda^{2-\alpha}(n\nabla \phi)\|_{L^2}\\
 &\leq \frac{1}{4} \|\Lambda^{\alpha+2} u\| _{L^2 }^2+
C\| n\|_{H^2}^2,
\end{aligned}
\end{equation}
where we have used
 $$
 \| \Lambda^{2-\alpha}(n\nabla \phi)\|_{L^2}\leq C\|n\nabla \phi\|_{L^2}^{\frac{\alpha}{2}}\|\nabla^2(n\nabla \phi)\|_{L^2}^{1-\frac{\alpha}{2}} \leq C\|\nabla\phi\|_{W^{2,\infty}}\|n\|_{H^2},
 $$
derived from \eqref{G-GN} and $1<2-\alpha<2$. The combination  of \eqref{eq3-18}, \eqref{eq3-20}, \eqref{eq3-21} and \eqref{eq3-210p} yields
 \begin{eqnarray*}
\begin{aligned}
\frac{d}{dt}\|\Delta u\| _{L^2 }^2
+\|\Lambda^{\alpha } \Delta u\| _{L^2 }^2
\leq 
C \left(1+\| u\|_{H^{\alpha+1}}^2\right)\|\Delta u\| _{L^2 }^2
+C\left(\|u\| _{L^2 }^2
+\|n\|_{H^2}^2
\right).
\end{aligned}
\end{eqnarray*}
Hence, thanks to Gr\"{o}nwall's inequality, \eqref{eq3-0}, \eqref{eq3-2}, \eqref{eq3-17} and $\|\Delta g\|_{L^2}\sim \|\nabla^2g\|_{L^2}$ for any $g\in \dot{H}^2$, there holds that, for any $0<T<T_*$,
\begin{equation*}
\begin{aligned}
&\sup_{t\in(0,T]}\|\nabla^2 u\| _{L^2 }+
\int_0^{T} \|\Lambda^{\alpha+2} u\| _{L^2 }^2\, dt\\
&~~\leq  C {\rm exp} \left\{\int_0^{T} 
\left(1+\|u\|_{L^2}^2+\|u\|_{H^{\alpha+1}}^2\right)\,dt
\right\}\\
&~~\quad \cdot\left(
\|\nabla^2 u_0\| _{L^2 }^2+ T
\sup_{t\in(0,T]}
\left(\|u\| _{L^2 }^2+\|n\|_{H^1}^2
\right)+\int_{0}^{T} \|\nabla^2 n\|_{L^2}^2\,dt\right)\leq  C(B_1,T_*),
\end{aligned}
\end{equation*}
which, as well as \eqref{eq3-14} and \eqref{eq3-17}, yields \eqref{eq3-10-1}.

\begin{itemize}
\item \textbf{Step 3: $H^2$-estimates of $n$.}
\end{itemize}

Finally, it suffices to show that, for any $0<T<T_*$,
\begin{equation}\label{eq3-21-1}
\sup_{t\in(0,T]}\|\nabla^2n\|_{L^2}^2+\int_{0}^{T}\|\nabla^3n\|_{L^2}^2\,dt\leq C(B_1,T_*).
	\end{equation}
To achieve it, we first apply $\Delta$  to $\eqref{eq1-1}_1$ and then perform the $L^2$-energy estimate to have 
\begin{equation}\label{eq3-23}
\begin{aligned}
	\frac{1}{2}\frac{d}{dt}
	\|\Delta n\| _{L^2 }^2+
	\|\nabla \Delta n\| _{L^2 }^2
	&= \int_{\mathbb{R}^3} \nabla\Delta n \cdot\nabla( u \cdot \nabla n+ \nabla\cdot(\chi(c) n \nabla c))\,dx\\
 &\leq \frac{1}{4}\|\nabla \Delta n\| _{L^2 }^2+C\|\nabla( u \cdot \nabla n)\|_{L^2}^2+ \|\nabla\nabla\cdot(\chi(c) n \nabla c))\|_{L^2}^2,
\end{aligned}	
\end{equation}
where the second term on the right-hand side of \eqref{eq3-23} can be estimated as
\begin{equation}\nonumber
\begin{aligned}
\|\nabla( u \cdot \nabla n)\|_{L^2}^2&\leq  \|\nabla u\|_{L^6}^2\|\nabla n\|_{L^{3}}^2+\|u\|_{L^{\infty}}^2\|\nabla^2n\|_{L^2}^2\\
&\leq C\|\nabla^2 u\|_{L^2}^2\|\nabla n\|_{H^1}^2+\|u\|_{H^2}^2\|\nabla^2n\|_{L^2}^2.
\end{aligned}
\end{equation}
Note that
\begin{equation}\label{nabla2nc}
\begin{aligned}
|\nabla\nabla\cdot(\chi(c) n \nabla c))|&\leq |\chi''(c)| |\nabla c|^3 n+2|\chi'(c)| |\nabla^2 c| |\nabla c| n+2|\chi'(c)| |\nabla c|^2 |\nabla n|\\
&\quad+|\chi(c)| |\nabla^2n| |\nabla c|+|\chi(c)| |\nabla n| |\nabla^2c|.
\end{aligned}
\end{equation}
By \eqref{nabla2nc} and the Gagliardo-Nirenberg inequality, we handle the last term on \eqref{eq3-23} as follows
\begin{equation}\nonumber
\begin{aligned}
\|\nabla \nabla\cdot(\chi(c) n \nabla c))\|_{L^2}^2& \leq C(1+\|c\|_{H^2}^3)\|n\|_{H^2}^2\|\nabla c\|_{H^2}^2.
\end{aligned}
\end{equation}
Thus, Gr\"{o}nwall's inequality, together with \eqref{eq3-0} and \eqref{eq3-10-1}, leads to \eqref{eq3-21-1}.

\vspace{3mm}

 The combination of \eqref{eq3-0}, \eqref{eq3-10-1} and \eqref{eq3-21-1} gives rise to the required bounds to extend the solution $(n,c,u)$ beyond $T_*$. Thus, we finish the proof of Theorem \ref{Th1.1}.

\subsection{Proof of Theorem \ref{Th1.2}}

Similarly, we show Theorem \ref{Th1.2} towards a contradiction. Assume that the maximal time $T_*$ is finite, and \eqref{eq1-5} fails, i.e., 
\begin{equation}
B_{2}:=\lim_{T\rightarrow T_*}\int_0^{T}\left(\|\nabla c\|_{L^{q_1}}^{p_1}+\|\nabla u\|_{L^{q_2}}^{p_2} \right)\,dt <+\infty \label{blow2}
\end{equation}
for any $(p_i, q_i) (i=1, 2)$ satisfying  \eqref{piqi1}. Our goal is to show that for any $0<T<T_*$, the $H^2$-norms of $(n, c, u)$ can be bounded by some constant depending on $T_*$ and $B_2$, and therefore the solution can be extended beyound $T_*$, which leads to the contradiction.

First, one deduces from $\eqref{eq1-1}_1$, \eqref{cinfty} and  \eqref{G-GN} that
\begin{equation}\nonumber
\begin{aligned}
	\frac{1}{2}\frac{d}{dt}\|n\|_{L^2 }^2+\|\nabla n\|_{L^2 }^2 
	&=
	\int_{\mathbb{R}^3}\chi(c) n\nabla c\cdot \nabla n\,dx\\
	&\leq   \sup_{0\leq c\leq \|c_0\|_{L^\infty }} |\chi (c)|  \|\nabla n\|_{L^2 } \|n\|_{L^\frac{2q_1}{q_1-2}}\|\nabla c\|_{L^{q_1} }\\
 &\leq C \|\nabla n\|_{L^2 }   \|n\|_{L^{2} }^\frac{q_1-3}{q_1} \|\nabla n\|_{L^2 }^\frac{3}{q_1} \|\nabla c\|_{L^{q_1} }\\
		& \leq  \frac{1}{2} \|\nabla n\|_{L^2 }^2+C\|\nabla c\|_{L^{q_1} }^{\frac{2q_1}{q_1-3}}\|n\|_{L^2 }^2,
\end{aligned}
	\end{equation}
which, together with $\frac{2q_1}{q_1-3}\leq p_1$ (see \eqref{piqi1}), leads to
	\begin{equation}\nonumber
	\frac{d}{dt}\|n\| _{L^2 }^2+\|\nabla n\|_{L^2 }^2	\leq C(1+\|\nabla c\|_{L^{q_1}}^{p_1})\|n\| _{L^2 }^2.
\end{equation}
Hence, it follows from \eqref{blow2} that
	\begin{equation}\label{nnnnn}
	\sup_{t\in(0,T]}\|n\|_{L^2}^2+\int_{0}^{T} \|\nabla n\|_{L^2}^2\,d\tau\leq e^{CT+CB_2}\|n_0\|_{L^2}^2
\end{equation}
for any $0<T<T_*$.

Then, recall that \eqref{eq3-4} holds. Arguing similarly as in  \eqref{eq3-222220}  in all the rangle $\alpha>\frac{1}{2}$, one has 
\begin{equation}\label{eq3-10111}
\begin{aligned}
& \frac{d}{dt}\|\nabla c\| _{L^2 }^2+\|\Delta c\| _{L^2 }^2\leq C\|(n, u)\| _{L^2 }^2+\frac{1}{8} \|\Lambda^{1+\alpha}u\|_{L^2}^2.
\end{aligned}
\end{equation}
We now deal with the $\dot{H}^1$-estimates of $u$. As in \eqref{eq3-6}, we have
\begin{equation}\nonumber
	\begin{aligned}
	\frac{1}{2} \frac{d}{dt} \|\nabla u\| _{L^2 }^2	+\|\Lambda ^{\alpha+1} u\| _{L^2 }^2
	=  \int_{\mathbb{R}^3} \Delta u \cdot n \nabla \phi\,dx+\int_{\mathbb{R}^3} \Delta u \cdot (u \cdot \nabla) u\,dx.
	\end{aligned}
\end{equation}
We estimate every term on the right-hand side of the above equality as follows. Arguing similarly as in  \eqref{case2uphi}, one has
\begin{equation}\nonumber
\begin{aligned}
\left|  \int_{\mathbb{R}^3} \Delta u \cdot n \nabla \phi\,dx\right|&\leq C\|\nabla u\| _{L^2 }^2+C\|n\| _{L^2 }^2+C\|\nabla  n\| _{L^2 }^2.
\end{aligned}
\end{equation}
After integrating by parts as in \eqref{eq3-8} and using $\div u=0$ and the Gagliardo–Nirenberg inequality \eqref{G-GN}, we deduce that
\begin{equation}\nonumber
\begin{aligned}
\left|\int_{\mathbb{R}^3} \Delta u \cdot (u \cdot \nabla) u\,dx	\right|&\leq   \|\nabla u\|_{L^3}^3\\
&\leq C  \|\nabla u\|_{L^{q_2}} \|\nabla u\|_{L^{\frac{2q_2}{q_2-1}}}^2\\
&\leq  C  \|\nabla u\|_{L^{q_2}} \|\nabla u\|_{L^{2}}^{2\theta} \|\Lambda^{\alpha}\nabla u\|_{L^2}^{2(1-\theta)}\\
&\leq \frac{1}{8} \|\Lambda^{\alpha+1}u\|_{L^2}^2+C\|\nabla u\|_{L^{q_2}}^{\frac{1}{\theta}} \|\nabla u\|_{L^2}^2,
\end{aligned}
\end{equation}
where $\theta=1-\frac{3}{2q_2\alpha}$ satisfies $\frac{1}{\theta}\leq p_2 $ due to \eqref{piqi1}. It thus holds that 
\begin{equation}\label{ummmmm}
	\begin{aligned}
	\frac{d}{dt} \|\nabla u\| _{L^2 }^2	+\|\Lambda ^{\alpha+1} u\| _{L^2 }^2\leq C(1+\|\nabla u\|_{L^{q_2}}^{p_2}) \|\nabla u\|_{L^2}^2+C\|n\| _{L^2 }^2+C\|\nabla  n\| _{L^2 }^2.
	\end{aligned}
\end{equation}
Adding \eqref{eq3-4}, \eqref{eq3-10111} and \eqref{ummmmm} together and then using Gr\"onwall's inequality, \eqref{blow2} and \eqref{nnnnn}, one has
\begin{eqnarray*}
	\begin{aligned}
	&\sup_{t\in(0,T]}(\|\nabla c\|_{L^2}^2+\| u\| _{H^1}^2)	+\int_{0}^{T}(\|\nabla^2 c\|_{L^2}^2+\|\Lambda ^{\alpha} u\| _{H^1}^2)\,dt\leq C(T_*,B_2). 
		\end{aligned}
\end{eqnarray*}

Finally, repeating the same arguments as in Subsection \ref{subsection31}, Steps 2-3, we are able to  establish the higher-order estimates
\begin{eqnarray*}
	\begin{aligned}
	&\sup_{t\in(0,T]}\|(n, c,u)\|_{H^2}^2	+\int_{0}^{T}(\|(\nabla n,\nabla c)\|_{H^2}^2+\|\Lambda ^{\alpha} u\| _{H^2}^2)\,dt\leq C(T_*,B_2),\quad\quad 0<T<T_*.
		\end{aligned}
\end{eqnarray*}
For the sake of simplicity, we omitted the details here.

Thus, based on the above estimates and Theorem \ref{Thlocal}, one can extend the solution $(n,c,u)$ beyond $T_*$. This contradicts the fact that $T_*$ is the maximal time for existence. The proof of Theorem \ref{Th1.2} is thus complete.

\section{Global existence and time-decay estimates}\label{section4}

In this section, we aim to prove some global existence results under some suitable initial assumptions. First, it is shown that when $\alpha$ is greater than or equal to Lions's index $\frac{5}{4}$, the local solution can be extend to a global one once $\|c_0\|_{L^{\infty}}$ is suitably small. Then, for the range $\frac{3}{4}<\alpha<\frac{5}{4}$, under a initial  mild small condition \eqref{smallness2}, we not only prove the global existence but also establish the uniform-in-time estimates, which enable us to study the large-time behavior of the solution.

\subsection{Proof of Theorem \ref{Th1.3}}\label{subsectionglobal1}



In order to prove \eqref{Th1.3}, we justify the blow-up criterion stated in Theorem \ref{Th1.2} as follows.

\begin{lemma}\label{lemma54blow}
Let $(n,c,u)$ be a strong solution to the Cauchy problem \eqref{eq1-1}-\eqref{eq1-2} on $[0,T_*)$ obtained in Theorem \ref{Thlocal}. Then there exists a constant $\delta_0=\delta_0(p)$ such that if $T_*$ is finite and \eqref{c0small} follows, then it holds for any $0<T<T_*$ that
\begin{equation}\label{nLp}
\begin{aligned}
\sup_{t\in(0,T]}\|n\|_{L^r}\leq C_{r}
\end{aligned}
\end{equation}
for any $r\in(1,\infty)$ and some uniform constant $C_r$.

Furthermore, we have
\begin{equation}\label{L2cu1}
\begin{aligned}
\int_0^{T}\|u\|_{L^{q}}^{p}\,dt\leq C_{T_*},
\end{aligned}
\end{equation}
where $C_{T_*}$ is a constant depending on $T_*$, and $p$, $q$ satisfy
\begin{align}
&\frac{2\alpha}{p}+\frac{3}{q}=\frac{3}{2},\quad 2\leq q\leq \frac{6}{3-2\alpha}\quad\text{for}\quad \frac{1}{2}<\alpha<\frac{3}{2},\quad 2\leq q\leq \infty \quad \text{for}\quad \alpha>\frac{3}{2}.\label{pq}
\end{align}
\end{lemma}

\begin{proof}
If $\|c_{0}\|_{L^{\infty}}\leq \delta^*$ with some suitably small $\delta^*>0$, then we are able to obtain \eqref{nLp} according to the coupling structure of $\eqref{eq1-1}_1$-$\eqref{eq1-1}_2$. The proof can be done by the following similar arguments as in the works \cite[Proposition 2]{ckl2014} or \cite[Lemma 3.1]{tao2011}. For the sake of simplicity, the details are omitted here.

 We then deal with the estimates of $u$.
Given the $L^2$-estimate of $n$ due to \eqref{nLp},  it follows from \eqref{eq3-4} that
\begin{equation}\label{u011}
\begin{aligned}
&\sup_{t\in[0,T]}\|u\|_{L^2}^2+2\int_{0}^{T}\|\Lambda^{\alpha}u\|_{L^2}^2\,dt\leq e^{T}( \|u_{0}\|_{L^2}^2+CC_{2}T).
\end{aligned}
\end{equation}
Hence, for $\theta=1-\frac{3}{\alpha}(\frac{1}{2}-\frac{1}{q})$ such that 
$$
\frac{1}{q}=\frac{\theta}{2}+(1-\theta)(\frac{1}{2}-\frac{\alpha}{3}),\quad\quad p(1-\theta)=2,
$$
one has
\begin{equation}\nonumber
\begin{aligned}
\int_0^{T}\|u\|_{L^{q}}^{p} \,dt&\leq C\int_0^{T}\|u\|_{L^{2}}^{p\theta} \|\Lambda^{\alpha}u\|_{L^2}^{p(1-\theta)} \,dt\leq C\sup_{t\in(0,T]} \|u\|_{L^{2}}^{p\theta}\int_{0}^{T}\|\Lambda^{\alpha}u\|_{L^2}^{2}\,dt\leq C_{T},
\end{aligned}
\end{equation}
derived from \eqref{u011}, the Gagliardo–Nirenberg inequality \eqref{G-GN}. This completes the proof of Lemma \ref{lemma54blow}.
\end{proof}

\paragraph{Proof of Theorem \ref{Th1.3}}

Suppose $\alpha\geq \frac{5}{4}$ and that $(n_0,c_0,u_0)$ satisfies the conditions in Theorem \ref{Th1.3}. By virtue of Theorem \ref{Thlocal}, there exists a maximal time $T_{*}>0$ such that a unique strong solution $(n,c,u)$ to the Cauchy problem \eqref{eq1-1}-\eqref{eq1-2} exists on $[0,T_*)$.

We claim $T_{*}=\infty$. Indeed, if $T_{*}<\infty$, then for any $0<T<T_*$, $(n,u)$ satisfies the estimates \eqref{nLp} and \eqref{L2cu1} obtained in Lemma \ref{lemma54blow}. In fact, for $(p,q)$ that satisfies \eqref{pq} such that \eqref{L2cu1} follows, we can take $(p_2,q_2)=(p,q)$ in \eqref{piqi} due to the fact that $\frac{3}{2}\leq 2\alpha-1$ for any $\alpha\geq \frac{5}{4}$. Therefore, the estimates \eqref{nLp} and \eqref{L2cu1} contradict the blow-up criterion in Theorem \ref{Th1.1}, so $(n,c,u)$ is a global strong solution to the Cauchy problem \eqref{eq1-1}-\eqref{eq1-2} that fulfills the properties \eqref{rThm1.3}.

\subsection{Proof of Theorem \ref{thmglobal}: Global existence}\label{subsectionglobal2}


In this subsection, we shall prove the global existence part in Theorem \ref{thmglobal} if the initial $L^2$ energy of $(n_0,c_0, u_0)$ is small but the highest-order norm of initial data can be arbitrarily large. Our proof is based on elaborate energy estimates, the bootstrap argument as well as interpolation inequalities.

To achieve the global existence, the key ingredient is to establish the following uniform a priori estimates.

\begin{proposition}\label{propapriori}
Let the assumptions of Theorem \ref{thmglobal} hold. Define
\begin{align}
\mathcal{E}_{0}:= \|(n_0, c_0,u_0)\|_{L^2}^2,\quad\quad  \mathcal{E}_{2}:= \|(\nabla^2 n_0,\nabla^2 c_0,\nabla^2 u_0)\|_{L^2}^2,\nonumber
\end{align}
and
\begin{align}
S_{f,\chi}:=\sup_{0\leq s\leq \|c_0\|_{L^{\infty}}}(|f(s)|+|f'(s)|+|\chi(s)|+|\chi'(s)|+|\chi''(s)|).\label{Sfchi}
\end{align}
For any $\frac{3}{4}<\alpha< \frac{5}{4}$, suppose that $(n,c,u)$ is a strong solution to the Cauchy problem   \eqref{eq1-1}-\eqref{eq1-2} defined on $[0,T)$ satisfying
\begin{equation}\label{apriori}
S_{f,\chi}^2 \|\nabla c\|_{L^3}^2+\|u\|_{L^3}^2+\int_{0}^{t}\|\nabla u\|_{L^2}^{\frac{4\alpha}{4\alpha-3}}\, d\tau\leq M_{0},
\end{equation}
for any $ 0<t<T$ and some generic constant $M_{0}>0$ given by \eqref{M0}, \eqref{M1}, \eqref{M2}, \eqref{M3} and \eqref{M4} below. Then, there exists a constant $\delta_{1}$ depending only on $S_{f,\chi}$, $\|(1+|x|)\nabla\phi\|_{W^{2,\infty}}$, $\||x|^{1+\alpha}\nabla\phi\|_{L^{\infty}}$ and $\|(\nabla^2n_0,\nabla^2c_0,\nabla^2 u_0)\|_{L^2}$ such that if $\eqref{smallness2}$ holds, then for any $0<t<T$ we have
\begin{equation}\label{apriori1}
\begin{aligned}
&S_{f,\chi}^2 \|\nabla c\|_{L^3}^2+\|u\|_{L^3}^2+\int_{0}^{t}\|\nabla u\|_{L^2}^{\frac{4\alpha}{4\alpha-3}} \,d\tau\leq \frac{1}{2}M_{0},
\end{aligned}
\end{equation}
\end{proposition}

\vspace{2mm}

The proof of Proposition \ref{propapriori} relies on the uniform estimates obtained in Lemmas \ref{lemmabasic}-\ref{lemmaH2} below.


\begin{lemma}\label{lemmabasic}
Under the assumptions of Proposition \ref{propapriori}, if \eqref{apriori} holds with some constant $M_0>0$, then for all $t\in(0,T)$, we have
\begin{equation}\label{basic0}
\begin{aligned}
&\sup_{\tau\in (0,t]}\|(n, c,u)\|_{L^2}^2+\int_{0}^{t}\|(\nabla n,\nabla c,\Lambda^{\alpha}u)\|_{L^2}^2\,d\tau\leq C_{1,\phi} \mathcal{E}_0,
\end{aligned} 
\end{equation}
for
$$
C_{1,\phi}:=C(1+\||x|^{1+\alpha}\nabla\phi\|_{L^{\infty}}^2).
$$
\end{lemma}

\begin{proof}
We recall that $c$ satisfies \eqref{ccL2} which is uniform in time. Taking the inner product of $\eqref{eq1-1}_{1}$ by $n$, we derive
\begin{equation}\nonumber
\begin{aligned}
&\frac{1}{2}\frac{d}{dt}\|n\|_{L^2}^2+\|\nabla n\|_{L^2}^2=\int_{\mathbb{R}^3} n\nabla n\cdot \nabla c \chi(c)\,dx.
\end{aligned} 
\end{equation}
According to the Sobolev embedding and $0\leq c\leq \|c_{0}\|_{L^{\infty}}$, one has
\begin{equation}\nonumber
\begin{aligned}
\left|\int_{\mathbb{R}^3} n\nabla n\cdot \nabla c \chi(c) \,dx\right|&\leq \sup_{0\leq s\leq \|c_{0}\|_{L^{\infty}}} \chi(s) \|\nabla n\|_{L^2} \|\nabla c\|_{L^3}\|n\|_{L^6}   \leq C_1 S_{f,\chi} \|\nabla c\|_{L^3} \|\nabla n\|_{L^2}^2,
\end{aligned} 
\end{equation}
where the constant $S_{f,\chi}$ is given by \eqref{Sfchi}. Therefore, we arrive at
\begin{equation}\nonumber
\begin{aligned}
&\frac{1}{2}\frac{d}{dt}\|n\|_{L^2}^2+\left(1-C_1 S_{f,\chi}\|\nabla c\|_{L^3}\right)\|\nabla n\|_{L^2}^2\leq 0,
\end{aligned} 
\end{equation}
where $C_1>0$ is a  generic constant. As long as \eqref{apriori} holds with 
\begin{align}
M_0\leq \frac{1}{(2C_1)^2},\label{M0}
\end{align}
remembering $\|\Delta c\|_{L^2}^2 \sim \|\nabla^2 c\|_{L^2}^2$, we are able to deduce that
\begin{equation}\label{nL2}
\begin{aligned}
&\sup_{\tau\in(0,t]}\|n\|_{L^2}^2+\int_0^t\|\nabla n\|_{L^2}^2\,d\tau\leq \|n_0\|_{L^2}^2.
\end{aligned} 
\end{equation}

On the other hand, from  $\eqref{eq1-1}_{3}$ and Hardy's inequality in Lemma \ref{lemmahardy}, one has
\begin{equation}\nonumber
\begin{aligned}
\frac{1}{2}\frac{d}{dt}\|u\|_{L^2}^2+\|\Lambda^\alpha u\|_{L^2}^2&=-\int_{\mathbb{R}^3} n\nabla \phi \cdot u\,dx\\
&\leq \||x|^{1+\alpha}\nabla\phi\|_{L^{\infty}}\Big\|\frac{n}{|x|}\Big\|_{L^2}\Big\|\frac{u}{|x|^{\alpha}}\Big\|_{L^{2}}\\
&\leq \frac{1}{2}\|\Lambda^{\alpha}u\|_{L^2}^2+C\||x|^{1+\alpha}\nabla\phi\|_{L^{\infty}}^2\|\nabla n\|_{L^{2}}^2,
\end{aligned} 
\end{equation}
which, as well as \eqref{nL2}, yields
\begin{equation}\nonumber
\begin{aligned}
&\|u\|_{L^2}^2+\int_{0}^{t}\|\Lambda^{\alpha}u\|_{L^2}^2\,d\tau\leq \|u_{0}\|_{L^2}+C\||x|^{1+\alpha}\nabla\phi\|_{L^{\infty}}^2\|n_0\|_{L^2}^2.
\end{aligned} 
\end{equation}
Hence, we end up with \eqref{basic0}.


\end{proof}


\begin{lemma}\label{lemmahigher}
Under the assumptions of Proposition \ref{propapriori}, if we have \eqref{apriori} with some constant $M_0>0$ and let $\mathcal{E}_0\leq 1$, then for all $t\in(0,T)$, it holds that 
\begin{equation}\label{h1}
\begin{aligned}
\sup_{\tau\in(0,t]}&\|(\nabla n,\nabla c,\nabla u)\|_{L^2}^2+\int_{0}^{t}(\|(\nabla^2n,\nabla^2 c,\nabla\Lambda^{\alpha}u)\|_{L^2}^2)\,d\tau\\
&\leq \mathcal{C}_{2,\phi,f,\chi} (1+\mathcal{E}_2 ) \mathcal{E}_0^{\frac{1}{2}},
\end{aligned} 
\end{equation}
where $\mathcal{C}_{2,\phi,f,\chi}$ is given by
\begin{equation}\nonumber
\begin{aligned}
\mathcal{C}_{2,\phi,f,\chi}:&=C(1+S_{f,\chi}^4)(1+\|(1+|x|^{1+\alpha})\nabla\phi\|_{L^{\infty}}^6)
\end{aligned}
\end{equation}
for some universal constant  $C>0$.
\end{lemma}

\begin{proof}
From $\eqref{eq1-1}_{2}$, it holds that
\begin{equation}\nonumber
\begin{aligned}
\frac{1}{2}\frac{d}{dt}\|\nabla c\|_{L^2}^2+\|\Delta c\|_{L^2}^2
\leq \|u\cdot\nabla c\|_{L^2}^2+\|n f(c)\|_{L^2}^2+\frac{1}{2}\|\Delta c\|_{L^2}^2.
\end{aligned}
\end{equation}
Regarding the first nonlinear term, one has
\begin{equation}\label{mmm1d}
\begin{aligned}
\|u\cdot\nabla c\|_{L^2}^2\leq \|u\|_{L^{3}}^2\|\nabla c\|_{L^6}^2\leq C_2\|u\|_{L^{3}}^2\|\Delta c\|_{L^2}^2
\end{aligned}
\end{equation}
for some constant $C_2>0$. Due to $f(0)=0$, we can write $f(c)=\tilde{f}(c)c$ with $\tilde{f}(c)=\int_{0}^{1}f'(\theta c)\,d\theta$. Therefore, the second nonlinear term is handled by
\begin{equation}\nonumber
\begin{aligned}
\|n f(c)\|_{L^2}^2&\leq S_{f,\chi}^2 \|n\|_{L^6}^2\|c\|_{L^3}^2\leq C S_{f,\chi}^2\|c\|_{L^3}^2\|\nabla n\|_{L^2}^2\\
&\leq C_2 S_{f,\chi}^2 \|c_0\|_{L^2}\|\nabla^2c_0\|_{L^2}\|\nabla n\|_{L^2}^2,
\end{aligned}
\end{equation}
where we noted $\|c\|_{L^3}\leq C \|c\|_{L^2}^{\frac{1}{3}}\|c\|_{L^{\infty}}^{\frac{2}{3}}\leq C\|c_0\|_{L^2}^{\frac{1}{3}}\|c_0\|_{L^{\infty}}^{\frac{2}{3}}\leq C\|c_0\|_{L^2}^{\frac{1}{2}}\|\nabla^2c_0\|_{L^2}^{\frac{1}{2}} $.
Thus, together with \eqref{cinfty}-\eqref{ccL2}, we arrive at
\begin{equation}\nonumber
\begin{aligned}
\frac{1}{2}\frac{d}{dt}\|\nabla c\|_{L^2}^2+(1-C_2\|u\|_{L^3}^2)\|\Delta c\|_{L^2}^2\leq C_2S_{f,\chi}^2\|c_0\|_{L^2}\|\nabla^2c_0\|_{L^2} \|\nabla n\|_{L^2}^2.
\end{aligned}
\end{equation}
Let $M_0$ satisfy 
\begin{equation}\label{M1}
\begin{aligned}
&M_0\leq \frac{1}{2C_2}.
\end{aligned}
\end{equation}
This leads to
\begin{equation}\label{nablace}
\begin{aligned}
\sup_{\tau\in(0,t]}\|\nabla c\|_{L^2}^2+\int_{0}^{t}\|\nabla^2 c\|_{L^2}^2\,d\tau&\leq \|\nabla c_0\|_{L^2}^2+ 2 C_2 M_0 \|c_0\|_{L^2}\|\nabla^2c_0\|_{L^2} \int_{0}^{t}\|\nabla n\|_{L^2}^2\,d\tau\\
&\leq \|\nabla c_0\|_{L^2}^2+2 C_2 M_0 \|c_0\|_{L^2}\|\nabla^2c_0\|_{L^2} C_{1,\phi}\mathcal{E}_0,
\end{aligned}
\end{equation}
where \eqref{apriori} and \eqref{basic0} have been used.

To establish the estimate of $n$, one takes the inner product of $\eqref{eq1-1}_{1}$ with $\Delta n$ to obtain
\begin{equation}\label{ddtnableu3o11}
\begin{aligned}
&\frac{1}{2}\frac{d}{dt}\|\nabla n\|_{L^2}^2+\left(\frac{3}{4}-C_3 (S_{f,\chi}^2\|\nabla c\|_{L^3}^2+\|u\|_{L^3}^2)  \right)\|\Delta n\|_{L^2}^2\\
&~\leq C_3 (S_{f,\chi}^2+S_{f,\chi}^4) \|\nabla^2c\|_{L^2}^4\|\nabla n\|_{L^2}^2,
\end{aligned} 
\end{equation}
where $C_3>0$ is a generic constant, and we have used Young's inequality, 
\begin{equation}\nonumber
\begin{aligned}
C\|u\cdot \nabla n\|_{L^2}^2\leq C\|u\|_{L^{3}}^2\|\nabla n\|_{L^6}^2\leq C_3\|u\|_{L^3}^2\|\Delta n\|_{L^2}^2,
\end{aligned} 
\end{equation}
and
\begin{equation}\nonumber
\begin{aligned}
&C\|\nabla (\chi(c) n \nabla c)\|_{L^2}^2\\
&\quad\leq CS_{f,\chi}^2(\|n\|_{L^\infty}^2\|\nabla^2 c\|_{L^2}^2+\|\nabla n\|_{L^6}^2\|\nabla c\|_{L^3}^2+\|\nabla c\|_{L^6}^4\|n\|_{L^6}^2)\\
&\quad\leq CS_{f,\chi}^2(\|\nabla n\|_{L^2}\|\nabla^2n\|_{L^{2}}\|\nabla^2 c\|_{L^2}^2+C\|\nabla^2n\|_{L^2}^2\|\nabla c\|_{L^3}^2+
\|\nabla^2 c\|_{L^2}^4\|\nabla n\|_{L^2}^2)\\
&\quad\leq \left(\frac{1}{4}+C_3 S_{f,\chi}^2\|\nabla c\|_{L^3}^2\right)\|\Delta n\|_{L^2}^2+C_3  (S_{f,\chi}^2+S_{f,\chi}^4) \|\nabla^2 c\|_{L^2}^4\|\nabla n\|_{L^2}^2.
\end{aligned} 
\end{equation}
We recall \eqref{apriori} and let $M_0$ be such that
\begin{equation}\label{M2}
\begin{aligned}
&M_0\leq \frac{1}{4C_3}.
\end{aligned}
\end{equation}
Thus, it follows from \eqref{basic0}, \eqref{nablace} and  \eqref{ddtnableu3o11} that
\begin{equation}\label{nablan}
\begin{aligned}
&\sup_{\tau\in(0,t]}\|\nabla n\|_{L^2}^2+\int_{0}^{t}\|\nabla^2 n\|_{L^2}^2\,d\tau\\
&\quad\leq \|\nabla n_0\|_{L^2}^2 +2C_3 (S_{f,\chi}^2+S_{f,\chi}^4 ) \sup_{\tau\in(0,t]}\|\nabla c\|_{L^2}^4 \int_0^t \|\nabla n\|_{L^2}^2\,d\tau\\
&\quad\leq \|\nabla n_0\|_{L^2}^2+2C_3(S_{f,\chi}^2+S_{f,\chi}^4 )(\|\nabla c_0\|_{L^2}^2+2C_2M_0 \|c_0\|_{L^2}\|\nabla^2c_0\|_{L^2} C_{1,\phi}\mathcal{E}_0)^2C_{1,\phi}\mathcal{E}_0.
\end{aligned} 
\end{equation}

Finally, we are ready to establish the $L^2$-estimates of $\nabla u$. Taking the inner product of $\eqref{eq1-1}_{3}$ with $\Delta u$, we have
\begin{equation}\label{Lambdau0}
\begin{aligned}
\frac{1}{2}\frac{d}{dt}\|\nabla u\|_{L^2}^2+\|\Lambda^{\alpha} \nabla u\|_{L^2}^2=\int_{\mathbb{R}^3} \nabla(n\nabla\phi):\nabla u\,dx-\int_{\mathbb{R}^3} \nabla(u\cdot\nabla u): \nabla u\,dx.
\end{aligned}
\end{equation}
We handle the first term on the right-hand side of \eqref{Lambdau0} in two cases. In the case $1\leq \alpha<\frac{5}{4}$, one deduces from Hardy's inequality \eqref{hardy} and the Gagliardo-Nirenberg inequality \eqref{G-GN} that
\begin{equation}\nonumber
\begin{aligned}
\left|\int_{\mathbb{R}^3} \nabla(n\nabla\phi): \nabla u\,dx\right|&=-\int_{\mathbb{R}^3} n\nabla\phi \cdot \Delta u\,dx\\
&\leq \||x|\nabla\phi\|_{L^{\infty}}\Big{\|}\frac{n}{|x|}\Big{\|}_{L^2}\|\Delta u\|_{L^2}\\
&\leq C\||x|\nabla\phi\|_{L^{\infty}}\|\nabla n\|_{L^2}\|u\|_{\dot{H}^{\alpha}}^{\alpha-1}\|u\|_{\dot{H}^{1+\alpha}}^{2-\alpha}.
\end{aligned}
\end{equation}
In the case $\frac{3}{4}<\alpha<1$, one also has
\begin{equation}\nonumber
\begin{aligned}
&\left|\int_{\mathbb{R}^3} \nabla(n\nabla\phi): \nabla u\,dx\right|\\
&\quad\leq \Big(\|\nabla n\|_{L^2}\|\nabla\phi\|_{L^{\infty}}+\Big\|\frac{n}{|x|}\Big\|_{L^2}\||x|\nabla\phi\|_{L^{\infty}}\Big)\|\nabla u\|_{L^2}\\
&\quad \leq C\|(1+|x|)\nabla\phi\|_{L^{\infty}}\|\nabla n\|_{L^2}\|u\|_{\dot{H}^{\alpha}}^{\alpha}\|u\|_{\dot{H}^{1+\alpha}}^{1-\alpha}.
\end{aligned}
\end{equation}
Hence, for all $\frac{3}{4}<\alpha<\frac{5}{4}$, we gain
\begin{equation}\label{111d}
\begin{aligned}
\left|\int_{\mathbb{R}^3} \nabla(n\nabla\phi): \nabla u\,dx\right|\leq \frac{1}{4}\|\Lambda^{\alpha} \nabla u\|_{L^2}^2+C\|(1+|x|)\nabla\phi\|_{L^{\infty}}^2\|\nabla n\|_{L^2}^2+C\|u\|_{\dot{H}^{\alpha}}^2.
\end{aligned}
\end{equation}
Now the second term on the right-hand side of \eqref{Lambdau0} is analyzed as follows. Using the duality of $\dot{H}^{\sigma}$ and $\dot{H}^{-\sigma}$ for any $\sigma\in\mathbb{R}$, we have
\begin{equation}\nonumber
\begin{aligned}
\Big| \int_{\mathbb{R}^3} \Delta u \cdot (u \cdot \nabla) u \,dx\Big|&\leq \|\Delta u\|_{\dot{H}^{\alpha-1}}\|u\cdot\nabla u\|_{\dot{H}^{1-\alpha}}\leq C\|\Lambda^{1+\alpha}u\|_{L^2}\|u\cdot\nabla u\|_{\dot{H}^{1-\alpha}}.
\end{aligned}	
\end{equation}
One deduces from $\nabla\cdot u=0$, the product law \eqref{product} and the interpolation inequality \eqref{G-GN} that
\begin{equation}\nonumber
\begin{aligned}
\|u\cdot\nabla u\|_{\dot{H}^{1-\alpha}}=\|\div(u\otimes u)\|_{\dot{H}^{1-\alpha}}\leq \|u\otimes u\|_{\dot{H}^{2-\alpha}}\leq C\|u\|_{\dot{H}^{\frac{7}{4}-\frac{\alpha}{2}}}^2\leq C\|u\|_{\dot{H}^1}^{3-\frac{3}{2\alpha}} \|u\|_{\dot{H}^{1+\alpha}}^{\frac{3}{2\alpha}-1}.
\end{aligned}
\end{equation}
Here we note that $\frac{3}{2\alpha}<2$ due to $\alpha>\frac{3}{4}$. It thus follows that
\begin{equation}\label{111d1}
\begin{aligned}
\left|\int_{\mathbb{R}^3} \nabla(u\cdot\nabla u): \nabla u\,dx\right|&\leq \|\Delta u\|_{\dot{H}^{\alpha-1}}\|u\cdot\nabla u\|_{\dot{H}^{1-\alpha}}\\
&\leq C\|u\|_{\dot{H}^1}^{3-\frac{3}{2\alpha}} \|u\|_{\dot{H}^{1+\alpha}}^{\frac{3}{2\alpha}}\\
&\leq \frac{1}{4} \|\Lambda^{\alpha} \nabla u\|_{L^2}^2+C\|\nabla u\|_{L^2}^{\frac{4\alpha}{4\alpha-3}}\|\nabla u\|_{L^2}^2.
\end{aligned}
\end{equation}
Putting \eqref{111d} and \eqref{111d1} into \eqref{Lambdau0} gives rise to
\begin{equation}\nonumber
\begin{aligned}
&\frac{d}{dt}\|\nabla u\|_{L^2}^2+\|\Lambda^{\alpha} \nabla u\|_{L^2}^2\leq C\|(1+|x|)\nabla\phi\|_{L^{\infty}}^2\|\nabla n\|_{L^2}^2+C\|u\|_{\dot{H}^{\alpha}}^2+C\|\nabla u\|_{L^2}^{\frac{4\alpha}{4\alpha-3}+2}.
\end{aligned}
\end{equation}
Integrating the above inequality over $[0,t]$, we infer that there exists a universal constant $C_4>0$ such that
\begin{equation}\nonumber
\begin{aligned}
&\sup_{\tau\in(0,t]}\|\nabla u\|_{L^2}^2+\int_{0}^{t}\|\Lambda^{\alpha} \nabla u\|_{L^2}^2\,d\tau\\
&\quad \leq \|\nabla u_{0}\|_{L^2}^2+C_4\|(1+|x|)\nabla\phi\|_{L^{\infty}}^2\int_{0}^{t}\|\nabla n\|_{L^2}^2\,d\tau+\int_{0}^{t}\|u\|_{\dot{H}^{\alpha}}^2\,d\tau\\
&\quad+C_4\int_{0}^{t}\|u\|_{\dot{H}^1}^{\frac{4\alpha}{4\alpha-3}} \,d\tau \sup_{\tau\in(0,t]}\|\nabla u\|_{L^2}^2.
\end{aligned}
\end{equation}
This, as well as \eqref{apriori} and  \eqref{basic0}, implies that
\begin{equation}\label{nablauL2}
\begin{aligned}
&\frac{1}{2}\sup_{t\in(0,\tau]}\|\nabla u\|_{L^2}^2+2\int_{0}^{t}\|\Lambda^{\alpha} \nabla u\|_{L^2}^2\,d\tau\\
&~\leq \|\nabla u_0\|_{L^2}^2+C_4(1+\|(1+|x|)\nabla\phi\|_{L^{\infty}}^2)C_{1,\phi}\mathcal{E}_0,
\end{aligned}
\end{equation}
when $M_0$ in \eqref{apriori} satisfies
\begin{align}
M_0\leq \frac{1}{2C_4}.\label{M3}
\end{align}
Recall that $\mathcal{E}_0\leq 1$. Combining \eqref{nablace}, \eqref{nablan}, \eqref{nablauL2} and the fact that $\|(\nabla n_0,\nabla c_0,\nabla u_0)\|_{L^2}^2\leq C\mathcal{E}_0^{\frac{1}{2}}\mathcal{E}_2^{\frac{1}{2}}$ due to \eqref{G-GN}, we conclude \eqref{h1} and thus complete the proof of Lemma \ref{lemmahigher}.
\end{proof}

\begin{lemma}\label{lemmaH2}
Under the assumptions of Proposition \ref{propapriori}, if \eqref{apriori} holds with some constant $M_0>0$ and let $\mathcal{E}_0\leq 1$, then for all $t\in(0,T)$, we have
\begin{equation}\label{HH2}
\begin{aligned}
\sup_{\tau\in(0,t]}&\|(\nabla^2 c,\nabla^2 u)\|_{L^2}^2+\int_{0}^{t}\|(\nabla^3 c,\nabla^2\Lambda^{\alpha}u)\|_{L^2}^2\,d\tau\leq C(\mathcal{E}_2+\mathcal{C}_{3,\phi,f,\chi}(1+\mathcal{E}_2^{3})),
\end{aligned} 
\end{equation}
and
\begin{equation}\label{HH2n}
\begin{aligned}
\sup_{\tau\in(0,t]}&\|\nabla^2 n\|_{L^2}^2+\int_{0}^{t}\|\nabla^3n\|_{L^2}^2\,d\tau\leq Ce^{\mathcal{C}_{3,\phi,f,\chi}(1+\mathcal{E}_2^{8})}\mathcal{E}_2,
\end{aligned} 
\end{equation}
where $C$ is a generic constant, and $\mathcal{C}_{3,\phi,f,\chi}$ denotes some constant depending only on $S_{\chi,f}$, $\|(1+|x|^{1+\alpha})\nabla\phi\|_{L^{\infty}}$, $\|\nabla^2\phi\|_{L^{\infty}}$ and $\| |x|\nabla^3\phi\|_{L^{\infty}}$.
\end{lemma}

\begin{proof}
We first perform the $L^2$-estimates of $\nabla^2 c$. One deduces from $\eqref{eq1-1}_{2}$ that
\begin{equation}\nonumber
\begin{aligned}
\frac{1}{2} \frac{d}{dt}&\|\Delta c\| _{L^2 }^2+\|\nabla \Delta c\| _{L^2 }^2\\
&= \int_{\mathbb{R}^3}\nabla\Delta c \cdot\nabla(n f(c))\,dx+\int_{\mathbb{R}^3}\nabla \Delta c \cdot\nabla(u\cdot \nabla c)\,dx\\
&\leq \frac{1}{16}\|\nabla \Delta c\| _{L^2 }^2+C\| \nabla\cdot(n f(c))\|_{L^2}^2+C\|\nabla \cdot(u\cdot \nabla c)\|_{L^2}^2.
\end{aligned} 
\end{equation}
Here by $f(0)=0$ and \eqref{G-GN} one can verify that
\begin{equation}\nonumber
\begin{aligned}
\| \nabla\cdot(n f(c))\|_{L^2}^2&\leq C S_{f,\chi}^2\Big(\|c\|_{L^6}^2\|\nabla n\|_{L^3}^2+\|n\|_{L^{\infty}}^2\|\nabla c\|_{L^2}^2\Big)\\
&\leq C S_{f,\chi}^2 \Big( \|c\|_{L^6}^2\|\nabla n\|_{L^2}\|\Delta n\|_{L^2}+ \|\nabla n\|_{L^2}\|\nabla^2 n\|_{L^2} \|\nabla c\|_{L^2}^2\Big) \\
&\leq C S_{f,\chi}^2\|\Delta n\|_{L^2}^2+C S_{f,\chi}^2\|\nabla c\|_{L^2}^4 \|\nabla n\|_{L^2}^2,
\end{aligned} 
\end{equation}
and
\begin{equation}\nonumber
\begin{aligned}
\|\nabla \cdot(u\cdot \nabla c)\|_{L^2}^2&\leq \|\nabla u\|_{L^{2}}^2\|\nabla c\|_{L^{\infty}}^2+\|u\|_{L^{6}}^2\|\nabla^2c\|_{L^3}^2\\
&\leq C\|\nabla u\|_{L^2}^2\|\nabla^2 c\|_{L^2}\|\nabla^3c\|_{L^2} \\
&\leq \frac{1}{2}\|\nabla \Delta c\|_{L^2}^2+C\|\nabla u\|_{L^2}^4\|\nabla^2 c\|_{L^2}^2.
\end{aligned} 
\end{equation}
Here $S_{f,\chi}$ is given by \eqref{Sfchi}. Thus, it follows that
\begin{equation}\nonumber
\begin{aligned}
&\frac{d}{dt}\|\Delta c\| _{L^2 }^2+ \|\nabla \Delta c\| _{L^2 }^2\\
&\leq CS_{f,\chi}^2\|\Delta n\|_{L^2}^2+C S_{f,\chi}^2\|\nabla c\|_{L^2}^4 \|\nabla n\|_{L^2}^2+C\|\nabla u\|_{L^2}^4\|\nabla^2 c\|_{L^2}^2,
\end{aligned} 
\end{equation}
which, together with \eqref{basic0} and \eqref{h1}, leads to
\begin{equation}\label{nabla2c}
\begin{aligned}
\sup_{\tau\in(0,t]}\|\nabla^2c\|_{L^2}^2+\int_0^t \|\nabla^3c\|_{L^2}^2\, d\tau
&\leq \|\nabla^2c_0\|_{L^2}^2+C S_{f,\chi}^2\int_{0}^{t}\|\nabla^2 n\|_{L^2}^2\,d\tau\\
&\quad+C S_{f,\chi}^2\sup_{\tau\in(0,t]}\|\nabla c\|_{L^2}^4\int_{0}^{t}\|\nabla n\|_{L^2}^2\,d\tau\\
&\quad+C\sup_{\tau\in(0,t]}\|\nabla u\|_{L^2}^4\int_{0}^{t}\|\nabla^2 c\|_{L^2}^2\,d\tau\\
&\leq \|\nabla^2c_0\|_{L^2}^2+CS_{f,\chi}^2\mathcal{C}_{2,\phi,f,\chi}(1+\mathcal{E}_2 ) \mathcal{E}_0^{\frac{1}{2}}\\
&\quad+CS_{f,\chi}^2  \mathcal{C}_{2,\phi,f,\chi}^2 C_{1,\phi}(1+\mathcal{E}_2 )^2\mathcal{E}_0^2\\
&\quad+C\mathcal{C}_{2,\phi,f,\chi}^3 (1+\mathcal{E}_2 )^3 \mathcal{E}_0^{\frac{3}{2}}.
\end{aligned} 
\end{equation}

Then, applying $\Delta u$ to $\eqref{eq1-1}_3$ and then taking the inner product of the resulting equation by $\Delta u$, we have
\begin{equation}\nonumber
\begin{aligned}
\frac{1}{2}\frac{d}{dt}&\|\Delta  u\|_{L^2}^2+\|\Lambda^{\alpha} \Delta u\|_{L^2}^2\\
&= \int_{\mathbb{R}^3}\nabla \Delta u \cdot \nabla ((u \cdot \nabla) u)\,dx-\int_{\mathbb{R}^3}\Delta u \cdot \Delta(n\nabla\phi)\,dx.
\end{aligned}
\end{equation}
Similarly to \eqref{eq3-20}, one can show
\begin{equation}\nonumber
\begin{aligned}
&\left|\int_{\mathbb{R}^3}\nabla \Delta u \cdot \nabla ((u \cdot \nabla) u)\,dx\right|\\
	&\quad\leq 3 \|\nabla u\| _{L^2 }
	\|\nabla^2 u\| _{L^2 }^{2-\frac{3}{2\alpha}}
	\|\Lambda^{2+\alpha} u\| _{L^2 }^{\frac{3}{2\alpha}}\\
 &\quad\leq \frac{1}{8}\|\nabla^{2}\Lambda^{\alpha} u\|_{L^2}^2+ C\|\nabla u\|_{L^2}^{\frac{4\alpha}{4\alpha-3}} \|\nabla^2 u\|_{L^2}^2.
\end{aligned}
\end{equation}
Moreover, due to $\alpha<2<2+\alpha$, one deduces from \eqref{G-GN} and \eqref{hardy} that
\begin{equation}\nonumber
\begin{aligned}
&\left|\int_{\mathbb{R}^3} \Delta u \cdot \Delta(n\nabla\phi)\,dx\right|\\
&\quad\leq \|\Delta u\|_{L^2}\left(\|\nabla^2 n\|_{L^2}\|\nabla\phi\|_{L^{\infty}}+2\|\nabla n\|_{L^2}\|\nabla^2\phi\|_{L^{\infty}}+\Big\|\frac{n}{|x|}\Big\|_{L^{2}}||x|\nabla^3 \phi\|_{L^{\infty}}\right) \\
&\quad\leq \frac{1}{8}\|\Lambda^{\alpha} \Delta u\|_{L^2}^2+C\|\Lambda^{\alpha} u\|_{L^2}^2\\
&\quad\quad +C\|\nabla\phi\|_{L^{\infty}}^2\|\nabla^2 n\|_{L^2}^2+C(\|\nabla^2\phi\|_{L^{\infty}}^2+\||x|\nabla^3 \phi\|_{L^{\infty}}^2)\|\nabla n\|_{L^2}^2.
\end{aligned}
\end{equation}
Therefore, we get
\begin{equation}\nonumber
\begin{aligned}
\frac{d}{dt}&\|\Delta u\|_{L^2}^2+\|\Lambda^{\alpha}\Delta u\|_{L^2}^2\\
&\leq C \|\nabla u\|_{L^2}^{\frac{4\alpha}{4\alpha-3}} \|\nabla^2 u\|_{L^2}^2+C\|\Lambda^{\alpha} u\|_{L^2}^2\\
&\quad +C(\|\nabla^2\phi\|_{L^{\infty}}^2+\||x|\nabla^3 \phi\|_{L^{\infty}}^2)\|\nabla n\|_{L^2}^2+C\|\nabla\phi\|_{L^{\infty}}^2\|\nabla^2 n\|_{L^2}^2,
\end{aligned}
\end{equation}
which, together with Gr\"onwall's inequality, \eqref{basic0}, \eqref{h1} and \eqref{apriori} with $M_0$ satisfying 
\begin{align}
M_0\leq 1,\label{M4}
\end{align}
 gives rise to
\begin{equation}\label{LLambdau0}
\begin{aligned}
&\sup_{\tau\in(0,t]}\|\nabla^2 u\|_{L^2}^2+\int_{0}^{t}\|\nabla^2\Lambda^{\alpha}u\|_{L^2}^2\,d\tau\\
&\quad\leq Ce^{\int_{0}^{t}\|\nabla u\|_{L^2}^{\frac{4\alpha}{4\alpha-3}} \,d\tau}\Big( \|\nabla^2 u_0\|_{L^2}^2+\int_{0}^{t}\|\Lambda^{\alpha} u\|_{L^2}^2\,d\tau\\
&\quad\quad+(\|\nabla^2\phi\|_{L^{\infty}}^2+\||x|\nabla^3 \phi\|_{L^{\infty}}^2)\int_{0}^{t}\|\nabla n\|_{L^2}^2\,d\tau\\
&\quad\quad+\|\nabla\phi\|_{L^{\infty}}^2\int_{0}^{t}\|\nabla^2 n\|_{L^2}^2\,d\tau\Big)\\
&\quad\leq C\bigg(\|\nabla^2u_0\|_{L^2}^2+(1+\|\nabla^2\phi\|_{L^{\infty}}^2+\| |x|\nabla^3\phi\|_{L^{\infty}}^2)C_{1,\phi}\mathcal{E}_0\\
&\quad\quad+\|\nabla \phi\|_{L^{\infty}}^2 \mathcal{C}_{2,\phi,f,\chi} \Big(1+\mathcal{E}_2)\mathcal{E}_0^{\frac{1}{2}} \bigg).
\end{aligned}
\end{equation}
Due to  \eqref{nabla2c} and \eqref{LLambdau0}, \eqref{HH2} follows.

Finally, it suffices to have the higher estimates of $n$. By $\eqref{eq1-1}_1$, \eqref{nabla2nc} and Gagliardo-Nirenberg-Sobolev inequalities, one has
\begin{equation}\nonumber
\begin{aligned}
	&\frac{1}{2}\frac{d}{dt}
	\|\Delta n\| _{L^2 }^2+
	\frac{1}{2}\|\nabla \Delta n\| _{L^2 }^2\\
 &\leq\|\nabla( u \cdot \nabla n+ \nabla\cdot(\chi(c) n \nabla c))\|_{L^2}^2\\
 &\leq C\|\nabla u\|_{L^3}^2\|\nabla n\|_{L^6}^2+C\|u\|_{L^{\infty}}^2\|\nabla^2 n\|_{L^2}^2\\
 &\quad+C S_{f,\chi}^2\Big( \|n\|_{L^{\infty}}^2\|\nabla c\|_{L^6}^6+\|n\|_{L^6}^2\|\nabla^2 c\|_{L^6}^2\|\nabla c\|_{L^6}^2+\|\nabla c\|_{L^6}^4\|\nabla n\|_{L^6}^2\Big)\\
 &\quad+C S_{f,\chi}^2\Big( \|\nabla^2 n\|_{L^3}^2\|\nabla  c\|_{L^{6}}^2+\|\nabla n\|_{L^3}^2\|\nabla^2 c\|_{L^6}^2+\|n\|_{L^{\infty}}^2\|\nabla^3c\|_{L^2}^2\Big)\\
 &\leq \frac{1}{4} \|\nabla \Delta n\|_{L^2}^2+CS_{f,\chi}^4\| \nabla^2 c\|_{L^2}^4 \|\nabla^2 n\|_{L^2}^2+C\|\nabla u\|_{L^2}\|\nabla^2 u\|_{L^2} \|\nabla^2 n\|_{L^2}^2\\
 &\quad +C S_{f,\chi}^2\Big( \|\nabla n\|_{L^2} \|\nabla^2 n\|_{L^2}\|\nabla^2 c\|_{L^2}^6+\|\nabla n\|_{L^2}^2\|\nabla^2 c\|_{L^2}^2\|\nabla^3 c\|_{L^2}^2+\|\nabla^2 c\|_{L^2}^4\|\nabla^2 n\|_{L^2}^2\Big)\\
 &\quad+C S_{f,\chi}^2\|\nabla n\|_{L^2} \|\nabla^2 n\|_{L^2} \|\nabla^3 c\|_{L^2}^2.
\end{aligned}	
\end{equation}
Using Gr\"onwall's inequality, we end up with
\begin{equation}\nonumber
\begin{aligned}
&\sup_{\tau\in(0,t]}\|\nabla^2 n\| _{L^2 }^2+\frac{1}{2}\int_{0}^{t}\|\nabla^3 n\| _{L^2 }^2\,d\tau\\
&\quad\leq e^{CS_{f,\chi}^4\sup\limits_{\tau\in(0,t]}\|\nabla^2c\|_{L^2}^2\int_0^t\|\nabla^2c\|_{L^2}^2\,d\tau}\bigg(\|\nabla^2n_0\|_{L^2}^2+C\sup_{\tau\in(0,t]}\|\nabla u\|_{L^2}\sup_{\tau\in(0,t]}\|\nabla^2u\|_{L^2}\int_{0}^t\|\nabla^2n\|_{L^2}^2\,d\tau\\
&\quad+C S_{f,\chi}^2 \sup_{\tau\in(0,t]}\|\nabla^2 c\|_{L^2}^6 \left(\int_0^t \|\nabla n\|_{L^2}^2\,d\tau\right)^{\frac{1}{2}}\left(\int_0^t \|\nabla^2 n\|_{L^2}^2\,d\tau\right)^{\frac{1}{2}}\\
&\quad+C S_{f,\chi}^2 \sup_{\tau\in(0,t]}\|\nabla n\|_{L^2}^2 \sup_{\tau\in(0,t]}\|\nabla^2 c\|_{L^2}^2 \int_0^t \|\nabla^3 c\|_{L^2}^2\,d\tau\\
&\quad+CS_{f,\chi}^2\sup_{\tau\in(0,t]} \|\nabla^2c\|_{L^2}^4 \int_0^t \|\nabla^2 n\|_{L^2}^2\,d\tau\\
&\quad+CS_{f,\chi}^2 \sup_{\tau\in(0,t]}   \|\nabla n\|_{L^2}\left(\int_{0}^{t}\|\nabla^2 n\|_{L^2}^2\,dt\right)^{\frac{1}{2}}\int_{0}^{t}\|\nabla^3 c\|_{L^2}^2\,dt\bigg).
\end{aligned}	
\end{equation}
Combining the above estimate with \eqref{h1}, \eqref{nabla2c} and \eqref{LLambdau0}, we obtain the desired estimates of $n$ in \eqref{HH2n} and thus complete the proof of Lemma \ref{lemmaH2}.

\end{proof}

\vspace{3mm}

\noindent\underline{\it\textbf{Proof of Proposition \ref{propapriori}:}}~
According to Theorem \ref{Th1.1}, the Cauchy problem   \eqref{eq1-1}-\eqref{eq1-2} has a unique strong solution $(n,c,u)$ on $[0,T_*)$, where $T_*>0$ is a maximal existence time. Let $M_{0}$ be a generic constant given by \eqref{M0}, \eqref{M1}, \eqref{M2}, \eqref{M3} and \eqref{M4}.

By virtue of Lemmas \ref{lemmabasic}-\ref{lemmahigher} and the Gagliardo–Nirenberg inequality in Lemma \ref{lemmaGN}, there holds that
\begin{equation}\nonumber
\begin{aligned}
\|u\|_{L^3}^2&\leq C\|u\|_{L^2}\|\nabla u\|_{L^2}\leq CC_{1,\phi}^{\frac{1}{2}}C_{2,\phi,f,\chi}^{\frac{1}{2}} (1+\mathcal{E}_2^{\frac{1}{2}})\mathcal{E}_0^{\frac{3}{4}}\leq \frac{1}{6}M_0,
\end{aligned} 
\end{equation}
provided that we let
\begin{align}
\mathcal{E}_0\leq \delta_1:=\min\Big\{1, \Big(6 C C_{1,\phi}^{\frac{1}{2}}C_{2,\phi}^{\frac{1}{2}}(1+\mathcal{E}_2^{\frac{1}{2}})\Big)^{-\frac{4}{3}}M_0\Big\}.\label{constant0}
\end{align}
In accordance with Lemmas \ref{lemmabasic} and \ref{lemmaH2} as well as the Gagliardo-Nirenberg inequality, we also have
\begin{equation}\nonumber
\begin{aligned}
S_{f,\chi}^2\|\nabla c\|_{L^3}^2&\leq CS_{f,\chi}^2\|c\|_{L^2}^{\frac{1}{4}}\|\nabla^2 c\|_{L^2}^{\frac{3}{4}}\\
&\leq  C S_{f,\chi}^2 C_{1,\phi}^{\frac{1}{4}} \Big((\mathcal{E}_2+\mathcal{C}_{3,\phi,f,\chi}(1+\mathcal{E}_2^{3}))\Big)^{\frac{3}{4}}\mathcal{E}_0^{\frac{1}{4}}\leq \frac{1}{6}M_0,
\end{aligned} 
\end{equation}
as long as
\begin{align}
\mathcal{E}_0\leq \delta_2:=\min\left\{1,\Big(6C S_{f,\chi}^2 C_{1,\phi}^{\frac{1}{4}}\Big((\mathcal{E}_2+\mathcal{C}_{3,\phi,f,\chi}(1+\mathcal{E}_2^{3}))\Big)^{\frac{3}{4}}\Big)^{-4}M_0^{4}\right\}.\label{constant}
\end{align}

To enclose the condition \eqref{apriori}, it suffices to justify
$$
\int_{0}^{t}\|\nabla u\|_{L^2}^{\frac{4\alpha}{4\alpha-3}}\, d\tau\leq \frac{1}{6}M_0.
$$
In the case $1\leq \alpha<\frac{5}{4}$, we observe that $\frac{4\alpha}{4\alpha-3}-2\alpha=\frac{2\alpha(5-4\alpha)}{4\alpha-3}>0$. We thence take advantage of the interpolation inequality \eqref{G-GN} and combine \eqref{basic0} and \eqref{h1}  together to obtain
\begin{equation}\nonumber
\begin{aligned}
\int_{0}^{t}\|\nabla u\|_{L^2}^{\frac{4\alpha}{4\alpha-3}}\, d\tau&\leq  C\sup_{\tau\in(0, t)}\|\nabla u\|_{L^2}^{\frac{2\alpha(5-4\alpha)}{4\alpha-3}}\int_{0}^{t} \|\nabla u\|_{L^2}^{2\alpha}\,d\tau\\
&\leq  \sup_{\tau\in(0,t)}\|\nabla u\|_{L^2}^{\frac{2\alpha(5-4\alpha)}{4\alpha-3} } \sup_{\tau\in(0,t)}\|u\|_{L^2}^{2(\alpha-1)}\int_{0}^{t} \|\Lambda^{\alpha}u\|_{L^2}^{2}\,d\tau\\
&\leq C\Big( \mathcal{C}_{2,\phi,f,\chi} ( 1+\mathcal{E}_2) \Big)^{\frac{\alpha(5-4\alpha)}{4\alpha-3} } \Big(\mathcal{C}_{1,\phi}\mathcal{E}_0\Big)^{\alpha}\\
&\leq C\mathcal{C}_{1,\phi}^{\alpha}\mathcal{C}_{2,\phi,f,\chi}^{\frac{\alpha(5-4\alpha)}{4\alpha-3} } (1+\mathcal{E}_2)^{\frac{\alpha(5-4\alpha)}{4\alpha-3} } \mathcal{E}_0^{\alpha}\\
&\leq \frac{1}{6}M_0,
\end{aligned}
\end{equation}
if for $1\leq \alpha<\frac{5}{4}$ we take
\begin{equation}\label{constant3}
\begin{aligned}
&\mathcal{E}_0\leq \delta_3=\min\Big\{ 1, \Big( 6C\mathcal{C}_{1,\phi}^{\alpha}\mathcal{C}_{2,\phi,f,\chi}^{\frac{\alpha(5-4\alpha)}{4\alpha-3} } (1+\mathcal{E}_2)^{\frac{\alpha(5-4\alpha)}{4\alpha-3} } \Big)^{-\frac{1}{\alpha}}M_0 \Big\}.
\end{aligned}
\end{equation}
As for the case $\frac{3}{4}<\alpha<1$, one has $\frac{4\alpha}{4\alpha-3}-2(1+\alpha)=\frac{2(3+\alpha-4\alpha^2)}{4\alpha-3}>0$. We thus conclude that
\begin{equation}\nonumber
\begin{aligned}
\int_{0}^{t}\|\nabla u\|_{L^2}^{\frac{4\alpha}{4\alpha-3}}\,d\tau&\leq  C\sup_{\tau\in(0,t)}\|\nabla u\|_{L^2}^{\frac{2(3+\alpha-4\alpha^2)}{4\alpha-3}}\int_{0}^{t} \|\nabla u\|_{L^2}^{2(1+\alpha)}\,d\tau\\
&\leq C\sup_{\tau\in(0,t)}\|\nabla u\|_{L^2}^{\frac{2(3+\alpha-4\alpha^2)}{4\alpha-3}}\sup_{\tau\in(0,t)}\|u\|_{L^2}^{2\alpha}\int_{0}^{t} \|\Lambda^{1+\alpha} u\|_{L^2}^{2}\,d\tau\\
&\leq C \Big( \mathcal{C}_{2,\phi,f,\chi} (1+\mathcal{E}_2) \Big)^{\frac{(3+\alpha-4\alpha^2)}{4\alpha-3}+1}\Big(\mathcal{C}_{1,\phi}\mathcal{E}_0\Big)^{\alpha}\\
&\leq C \mathcal{C}_{1,\phi}^{\alpha}\mathcal{C}_{2,\phi,f,\chi}^{\frac{\alpha(5-4\alpha)}{4\alpha-3} } (1+\mathcal{E}_2)^{\frac{\alpha(5-4\alpha)}{4\alpha-3} } \mathcal{E}_0^{\alpha}\\
&\leq \frac{1}{6}M_0,
\end{aligned}
\end{equation}
when 
\begin{equation}
\begin{aligned}
&\mathcal{E}_0\leq \delta _3:=\min\Big\{ 1, \Big( 6 C \mathcal{C}_{1,\phi}^{\alpha}\mathcal{C}_{2,\phi,f,\chi}^{\frac{\alpha(5-4\alpha)}{4\alpha-3} } (1+\mathcal{E}_2)^{\frac{\alpha(5-4\alpha)}{4\alpha-3} } \Big)^{-\frac{1}{\alpha}}M_0 \Big\}.\label{constant4}
\end{aligned}
\end{equation}

Based on above estimates, we choose $\mathcal{E}_0\leq \delta_0:=\min\{\delta_1,\delta_2,\delta_3\}$. Thus, the estimate \eqref{apriori1} is true for all $[0,T_*)$, and therefore the bounds in Lemma \ref{lemmabasic}-\ref{lemmahigher} are valid. A standard continuity argument implies $T_*=\infty$ and that  $(n,c,u)$ is indeed a global strong solution to the Cauchy problem   \eqref{eq1-1}-\eqref{eq1-2} satisfying the uniform estimates \eqref{uniformglobal}. To complete the proof of Theorem \ref{thmglobal}, we will establish the decay estimates \eqref{decay} in the next subsection.

\subsection{Proof of Theorem \ref{thmglobal}: Time-decay estimates}\label{subsectiondecay}

This subsection is devoted to the time-decay part in Theorem \ref{thmglobal}. Let $\frac{3}{4}<\alpha<\frac{5}{4}$. Under the assumptions of Theorem \ref{thmglobal}, we consider the global solution $(n,c,u)$ to the Cauchy problem \eqref{eq1-1}-\eqref{eq1-2} given by Subsection \ref{subsectionglobal2}. 

Under the conditions $n_0,u_0\in L^1$, integrating $\eqref{eq1-1}_1$ and $\eqref{eq1-1}_2$ over $[0,t]\times\mathbb{R}^3$ and using \eqref{cinfty}, we obtain 
\begin{align}
\|n\|_{L^1}=\|n_0\|_{L^1},\quad\quad \|c\|_{L^1}=\|c_0\|_{L^1},\quad\quad t\geq 0.\label{ncL1}
\end{align}

Now we show the $L^p$-decay estimates of $n$ and $c$. For all $2\leq p<\infty$, we perform the $L^p$-energy estimates for $\eqref{eq1-1}_1$ as follows
\begin{equation}\nonumber
\begin{aligned}
\frac{d}{dt}&\| n\|_{L^p}^p+\frac{4(p-1)}{p}\|\nabla n^{\frac{p}{2}}\|_{L^2}^2\\
&\leq p(p-1)\int_{\mathbb{R}^3} \chi(c)\nabla c n^{p-1}|\nabla n| \, dx\\
&\leq 2(p-1) \sup_{0\leq c\leq \|c_0\|_{L^\infty}} |\chi(c)|  \|\nabla c\|_{L^3}\|n^{\frac{p}{2}}\|_{L^6}\|\nabla n^{\frac{p}{2}}\|_{L^2}\\
&\leq C  \|\nabla c\|_{L^3}\|\nabla n^{\frac{p}{2}}\|_{L^2}^2.
\end{aligned}
\end{equation}
Since $\|\nabla c\|_{L^3}$ is suitably small due to \eqref{basic0}, \eqref{HH2} and
\begin{align}
    \|\nabla c\|_{L^3}&\leq C\|c\|_{L^2}^{\frac{1}{2}}\|\nabla^2 c\|_{L^2}^{\frac{1}{2}}<<1,\label{nablacL3}
\end{align}
it follows that
\begin{equation}
\begin{aligned}
&\frac{d}{dt}\| n\|_{L^p}^p+\frac{2(p-1)}{p}\|\nabla n^{\frac{p}{2}}\|_{L^2}^2\leq 0.\label{dissnine}
\end{aligned}
\end{equation}
In view of \eqref{ncL1} and the Gagliardo–Nirenberg inequality \eqref{G-GN}, we have 
\begin{equation}\label{dissn}
\begin{aligned}
\| n\|_{L^p}^{p}= \|n^{\frac{p}{2}}\|_{L^{2}}^2\leq C\|n^{\frac{p}{2}}\|_{L^{\frac{2}{p}}}^{\frac{4}{3p-1}}  \|\nabla n^{\frac{p}{2}}\|_{L^2}^{\frac{6(p-1)}{3p-1}}\leq C\|n_{0}\|_{L^1}^{\frac{8}{p(3p-1)}}\|\nabla n^{\frac{p}{2}}\|_{L^2}^{\frac{6(p-1)}{3p-1}}.
\end{aligned}
\end{equation}
Putting \eqref{dissn} into \eqref{dissnine} leads to the differential inequality
\begin{equation}\nonumber
\begin{aligned}
&\frac{d}{dt}\|n\|_{L^p}^p+C\Big(\|n\|_{L^p}^{p}\Big)^{1+\frac{2}{3(p-1)}} \leq 0,
\end{aligned}
\end{equation}
from which we infer
\begin{equation}
\begin{aligned}
\| n\|_{L^p}\leq C(1+ t)^{-\frac{3}{2}(1-\frac{1}{p})}\|n_{0}\|_{L^1},\quad\quad 2\leq p<\infty.\label{ndecay1}
\end{aligned}
\end{equation}
In addition, using H\"older's inequality gives
\begin{equation*}
\begin{aligned}
\| n\|_{L^p}\leq \|n\|_{L^1}^{\frac{2}{p}-1} \|n\|_{L^2}^{2-\frac{2}{p}}\leq C(1+t)^{-\frac{3}{2}(1-\frac{1}{p})},\quad\quad 1< p<2.\label{ndecay2}
\end{aligned}
\end{equation*}
A similar argument yields
\begin{equation}
\begin{aligned}
\| c\|_{L^p}\leq C(1+t)^{-\frac{3}{2}(1-\frac{1}{p})},\quad\quad 1< p<\infty.\label{cdecay1}
\end{aligned}
\end{equation}
The details are omitted.

Next, we are going to establish the higher-order decay estimates of $n$ and $c$. Multiplying $\eqref{eq1-1}_1$ by $\Delta n$ and integrating it by parts yield
\begin{equation}\nonumber
\begin{aligned}
\frac{1}{2}\frac{d}{dt}\|\nabla n\|_{L^2}^2+\|\Delta n\|_{L^2}^2
\leq \|u\cdot\nabla n\|_{L^2}^2+\|\nabla\cdot (\chi(c)n\nabla c)\|_{L^2}^2+\frac{1}{2}\|\Delta n\|_{L^2}^2,
\end{aligned}
\end{equation}
where the nonlinear terms are analyzed as 
\begin{equation}\nonumber
\begin{aligned}
 \|u\cdot\nabla n\|_{L^2}^2\leq \|u\|_{L^3}^2\|\nabla n\|_{L^6}^2\leq C\|u\|_{L^3}^2\|\Delta n\|_{L^2}^2,
\end{aligned}
\end{equation}
and
\begin{equation}\nonumber
\begin{aligned}
\|\nabla\cdot (\chi(c)n\nabla c)\|_{L^2}^2&\leq C\|n\|_{L^6}^2\|\nabla c\|_{L^6}^4+C\|\nabla c\|_{L^3}^2\|\nabla n\|_{L^6}^2+C\|n\|_{L^{\infty}}^2\|\nabla^2c\|_{L^2}^2\\
&\leq C(\|n\|_{L^6}^2\|\nabla^2 c\|_{L^2}^2+\|n\|_{L^6}\|\nabla^2n\|_{L^2})\|\Delta c\|_{L^2}^2+C\|\nabla c\|_{L^3}^2 \|\Delta n\|_{L^2}^2.
\end{aligned}
\end{equation}
Hence, there holds that
\begin{equation}\label{ddtnablandecay}
\begin{aligned}
&\frac{d}{dt}\|\nabla n\|_{L^2}^2+(1-C\|(\nabla c,u)\|_{L^3}^2)\|\Delta n\|_{L^2}^2\\
&\leq C(\|n\|_{L^6}^2\|\nabla^2 c\|_{L^2}^2+\|n\|_{L^6}\|\nabla^2n\|_{L^2}) \|\Delta c\|_{L^2}^2.
\end{aligned}
\end{equation}
This requires an estimate of $c$ at the $\dot{H}^1$-regularity level. To achieve it, one deduces from $\eqref{eq1-1}_2$ that
\begin{equation}\nonumber
\begin{aligned}
\frac{1}{2}\frac{d}{dt}\|\nabla c\|_{L^2}^2+\|\Delta c\|_{L^2}^2
\leq \|u\cdot\nabla c\|_{L^2}^2-\int_{\mathbb{R}^3} n f(c) \Delta c \,  dx+\frac{1}{2}\|\Delta c\|_{L^2}^2.
\end{aligned}
\end{equation}
The first nonlinear term has been estimated in \eqref{mmm1d}. We recall that $f(c)=\tilde{f}(c)c$ with $\tilde{f}(c)=\int_{0}^{1}f'(\theta c)\,d \theta$ due to $f(0)=0$. Therefore, the second nonlinear term can be handled by
\begin{equation}\nonumber
\begin{aligned}
\left|\int n f(c) \Delta c\, dx\right|&\leq \left| \int_{\mathbb{R}^3} \tilde{f}(c)c\nabla n \cdot \nabla c\, dx\right|+\left|\int_{\mathbb{R}^3} n f'(c)\nabla c\cdot \nabla  c \, dx\right|\\
&\leq C\|\nabla n\|_{L^6}\|c\|_{L^{\frac{3}{2}}}\|\nabla c\|_{L^6}+C\|n\|_{L^{\frac{3}{2}}}\|\nabla c\|_{L^{6}}^2\\
&\leq C\|(n,c)\|_{L^{\frac{3}{2}}}\|\Delta c\|_{L^2}^2+C\|c\|_{L^{\frac{3}{2}}}\|\Delta n\|_{L^2}^2.
\end{aligned}
\end{equation}
It thus holds that
\begin{equation}\label{ddtnablacdecay}
\begin{aligned}
\frac{d}{dt}\|\nabla c\|_{L^2}^2+(1-C\|(n,c)\|_{L^{\frac{3}{2}}}-C\|u\|_{L^3}^2)\|\Delta c\|_{L^2}^2\leq C\|c\|_{L^{\frac{3}{2}}}\|\Delta n\|_{L^2}^2.
\end{aligned}
\end{equation}
Adding \eqref{ddtnablandecay} and \eqref{ddtnablacdecay} together, we get
\begin{equation}\nonumber
\begin{aligned}
&\frac{d}{dt}\|(\nabla n,\nabla c)\|_{L^2}^2\\
&\quad+\left(1-C\|(n,c)\|_{L^{\frac{3}{2}}}-C\|(\nabla c,u)\|_{L^3}^2-C\|n\|_{L^6}^2\|\nabla^2 c\|_{L^2}^2-C\|n\|_{L^6}\|\nabla^2n\|_{L^2}\right)\|(\Delta n,\Delta c)\|_{L^2}^2\leq 0.
\end{aligned}
\end{equation}
Recall that in Lemma \ref{lemmabasic}, the $L^2$-norm of $\|(n, c, u)\|_{L^2}^2$ can be bounded by $C_{1, \phi}\mathcal{E}_0$ which can be suitably small, while Lemma \ref{lemmaH2} implies that the norms $\|\nabla^2n\|_{L^2}$ and $\|\nabla^2c\|_{L^2}$ are uniformly bounded in time. Furthermore, we have \eqref{nablacL3} and
\begin{equation}\nonumber
\begin{aligned}
\|(n,c)\|_{L^{\frac{3}{2}}}&\leq C\|(n,c)\|_{L^1}^{\frac{1}{3}}\|(n,c)\|_{L^2}^{\frac{2}{3}}<<1,\\
\|n\|_{L^6}&\leq C\|n\|_{L^2}^{\frac{1}{2}} \|\nabla^2n\|_{L^2}^{\frac{1}{2}}<<1,\\
\|u\|_{L^3}&\leq C\|u\|_{L^2}^{\frac{3}{4}}\|\nabla^2 u\|_{L^2}^{\frac{1}{4}}<<1.
\end{aligned}
\end{equation}
Therefore,  we obtain the energy inequality
\begin{equation}\nonumber
\begin{aligned}
&\frac{d}{dt}\|(\nabla c,\nabla n)\|_{L^2}^2+\|(\Delta c,\Delta n)\|_{L^2}^2\leq 0.
\end{aligned}
\end{equation}
The Gagliardo–Nirenberg inequality \eqref{G-GN} and \eqref{ncL1} ensure that
\begin{equation}\nonumber
\begin{aligned}
\|(\nabla c,\nabla n)\|_{L^2}\leq C \|(c,n)\|_{L^1}^{\frac{2}{7}} \|(\Delta c,\Delta n)\|_{L^2}^{\frac{5}{7}}\leq C \|(\Delta c,\Delta n)\|_{L^2}^{\frac{5}{7}}.
\end{aligned}
\end{equation}
Thus, we get
\begin{equation}\nonumber
\begin{aligned}
&\frac{d}{dt}\|(\nabla c,\nabla n)\|_{L^2}^2+C\Big(\|(\nabla c,\nabla n)\|_{L^2}^2 \Big)^{1+\frac{2}{5}}\leq 0.
\end{aligned}
\end{equation}
Solving this differential inequality, we have the decay of $(\nabla c,\nabla n)$ as follows
\begin{equation}\nonumber
\begin{aligned}
\|(\nabla c,\nabla n)\|_{L^2}\leq C(1+t)^{-\frac{5}{4}}.\label{decaycnnabla}
\end{aligned}
\end{equation}

We finally deal with the estimates of $u$. Taking the inner product of $\eqref{eq1-1}_3$ with $\Lambda^{2\alpha}u$ leads to
\begin{equation}\label{ddtlambda}
\begin{aligned}
\frac{1}{2}\frac{d}{dt}\|\Lambda^{\alpha}u\|_{L^2}^2+\|\Lambda^{2\alpha}u\|_{L^2}^2
\leq \frac{1}{2}\|\Lambda^{2\alpha}u\|_{L^2}^2+\|u\cdot\nabla u\|_{L^2}^2+\|n\nabla\phi\|_{L^2}^2.
\end{aligned}
\end{equation}
Multiplying \eqref{ddtlambda} by $t$, we arrive at 
\begin{equation}\nonumber
\begin{aligned}
\frac{d}{dt}\Big(t\|\Lambda^{\alpha}u\|_{L^2}^2\Big)+t\|\Lambda^{2\alpha}u\|_{L^2}^2\leq t \|u\cdot\nabla u\|_{L^2}^2+t\|n\nabla\phi\|_{L^2}^2.
\end{aligned}
\end{equation}
By virtue of the product law \eqref{product} with $s_1=\frac{3}{2}-\alpha$, $s_{2}=\alpha$ and $\alpha<\frac{5}{2}-\alpha<1+\alpha$, it holds that
\begin{equation}\nonumber
\begin{aligned}
t \|u\cdot\nabla u\|_{L^2}^2\leq C\|u\|_{\dot{H}^{\frac{5}{2}-\alpha}}^2 t\|\Lambda^{\alpha}u\|_{L^2}^2\leq C\|u\|_{\dot{H}^{\alpha}\cap\dot{H}^{1+\alpha}}^2 t\|\Lambda^{\alpha}u\|_{L^2}^2.
\end{aligned}
\end{equation}
And Hardy's inequality in Lemma \ref{lemmahardy} guarantees that
\begin{equation}\nonumber
\begin{aligned}
t\|n\nabla\phi\|_{L^2}^2\leq \||x|\nabla\phi\|_{L^{\infty}}^2 t\Big\|\frac{n}{|x|}\Big\|_{L^2}^2\leq C t\|\nabla n\|_{L^2}^2\leq C(1+t)^{-\frac{3}{2}}.
\end{aligned}
\end{equation}
Hence, it follows that
\begin{equation}\nonumber
\begin{aligned}
\frac{d}{dt}\Big(t\|\Lambda^{\alpha}u\|_{L^2}^2\Big)+t\|\Lambda^{2\alpha}u\|_{L^2}^2\leq C\|u\|_{\dot{H}^{\alpha}\cap\dot{H}^{1+\alpha}}^2 t\|\Lambda^{\alpha}u\|_{L^2}^2+C(1+t)^{-\frac{3}{2}}.
\end{aligned}
\end{equation}
Employing Gr\"onwall's inequality yields
\begin{equation}\label{udecay1}
\begin{aligned}
&t\|\Lambda^{\alpha}u\|_{L^2}^2+\int_{0}^{t}\tau \|\Lambda^{2\alpha}u\|_{L^2}^2\,d\tau\leq Ce^{\int_{0}^{t}\|u\|_{\dot{H}^{\alpha}\cap\dot{H}^{1+\alpha}}^2 \,d\tau}\int_{0}^{t}(1+\tau)^{-\frac{3}{2}}d\tau\leq C.
\end{aligned}
\end{equation}
which, together with Sobolev's embedding theorem, gives
\begin{equation*}\label{udecay2}
\begin{aligned}
&\|u\|_{L^{\frac{6}{3-2\alpha}}}\leq \|\Lambda^{\alpha}u\|_{L^2}\leq C(1+t)^{-\frac{1}{2}}.
\end{aligned}
\end{equation*}
Furthermore, one also has
\begin{equation}\label{udecay3}
\begin{aligned}
&\|u\|_{L^p}\leq \|u\|_{L^2}^{1-\frac{3}{\alpha}(\frac{1}{2}-\frac{1}{p})} \|u\|_{L^{\frac{6}{3-2\alpha}}}^{\frac{3}{\alpha}(\frac{1}{2}-\frac{1}{p})}\leq C(1+t)^{-\frac{3}{2\alpha}(\frac{1}{2}-\frac{1}{p})},\quad 2<p<\frac{6}{3-2\alpha}.
\end{aligned}
\end{equation}
By \eqref{ndecay1}-\eqref{cdecay1} and \eqref{udecay1}-\eqref{udecay3}, we conclude \eqref{decay} and complete the proof of Theorem \ref{thmglobal}.

\section{Conclusion and extensions}\label{sectionextension}

In this work, we first study the mechanism of possible finite time blow-up of strong solutions to the Cauchy problem of the generalized chemotaxis-Navier-Stokes system in spatially three dimensions. Based on these criteria and some uniform a-priori estimates, we then prove some new global existence results. In particular, we establish uniform-in-time evolution and large-time behavior of global solutions for initial data with small $L^2$ energy and possibly large oscillations. We discuss some possible extensions and questions below.

\vspace{2mm}

\noindent
(1) {\emph{Fractional diffusion effect on the population density equation}}. Our method may be applied to the study of the following chemotaxis-Navier-Stokes system:
\begin{equation}\nonumber
	\begin{cases}
\partial_t n+u\cdot \nabla n=-(-\Delta)^{\beta} n-
\nabla \cdot (\chi(c)n \nabla c),\\
\partial_t c+u \cdot \nabla c=\Delta c-nf(c),\\
\partial_t u +u \cdot \nabla u+\nabla P=-(-\Delta)^\alpha u-n\nabla \phi,\\
\nabla \cdot u=0.
	\end{cases}
\end{equation}
Here the exponent $\beta$ will influence the blow-up mechanism and regularity estimates of the population density $n$. We expect that a result similar to Theorem \ref{Th1.3} holds without the smallness of the $L^{\infty}$-norm for $c_0$ in \eqref{c0small} when $\beta$ is suitably large.

\vspace{3mm}

\noindent
(2) {\emph{Global weak solutions}}. The global existence of weak solutions to \eqref{eq1-1}-\eqref{eq1-2} is well expected. However, the uniqueness of weak solutions is open, which is similar to the classical Navier-Stokes equations. It is possible to study the uniqueness and regularity issues assuming the quantities in \eqref{eq1-3} or \eqref{eq1-5} are bounded in a finite time-interval. Furthermore, it also would be interesting to study the large-time behavior of global werak solutions to the Cauchy problem. To this end, one may apply the Fourier spilitting method, cf. e.g., \cite{BS,Wiegner-JLMS_1987}

\vspace{3mm}

\noindent
(3) {\emph{Enhanced dissipation phenomena.}} When the solution is close to some non-trivial stationary state, for example, the Couette flow, the enhanced dissipation  phenomena have been studied in the recent works \cite{ZZZ-2021-JFA,He-2023-SIAM} and references therein. It would be interesting to adapt the methodology in \cite{ZZZ-2021-JFA,He-2023-SIAM} to be able to deal with the system \eqref{eq1-1} even in 2D. We also expect that this mechanism can relax the smallness of the $L^2$-norm for initial data in \eqref{smallness2} using the argument in the proof of Theorem \ref{thmglobal}.

\section{Appendix: Local well-posedness}\label{section5}

In this appendix, we prove Theorem \ref{Thlocal} concerning the local well-posedness to \eqref{eq1-1}-\eqref{eq1-2} with general initial data. Our work can also cover all the dissipation exponent $\alpha>\frac{1}{2}$. In particular, we extend the previous works~\cite{duan2010,cpx2023} on  the local well-posedness for the chemotaxis-Navier-Stokes system in the case $\alpha=1$, without assuming that the initial data is suitably small or higher order regularity of $c_0$.

\vspace{3mm}

\noindent\underline{\bf\emph{Proof of Theorem~\ref{Thlocal}}} We split the proof into four steps.

\begin{itemize}

\item \bf{Step 1: Construction of approximate sequence}
\end{itemize}

Set $(n^0,c^0,u^0,P^0)=(0,0,0,0)$. We consider the following iterative approximate system for $j\geq0$: 
\begin{equation}\label{NSn}
	\begin{cases}
\partial_t n^{j+1}-\Delta n^{j+1}=-u^j\cdot \nabla n^{j}-
\nabla \cdot (\chi(c^j)n^{j} \nabla c^{j}),&\\
\partial_t c^{j+1}-\Delta c^{j+1}=-u^j\cdot \nabla c^{j}-n^{j}f(c^j),&\\
\partial_t u^{j+1}+(-\Delta)^\alpha u^{j+1}+\nabla P^{j+1}=-u^{j} \cdot \nabla u^{j}-n^j\nabla \phi,&\\
\nabla \cdot u^{j+1}=0
	\end{cases}
\end{equation}
with
\[(n^{j+1}(x,0),c^{j+1}(x,0),u^{j+1}(x,0))=(n_0(x),c_0(x),u_0(x)).
\]

Notice that the system~\eqref{NSn} is linear with respect to $(n^{j+1}, c^{j+1}, u^{j+1})$, so the existence and uniqueness of solutions are evident. Hence, by induction, we will verify  uniform regularity estimates of the  sequence $\{(n^j,c^j,u^j, P^j)\}_{j=0}^\infty$ in $t\in(0,T_0]$ for a short time $T_0$  determined in Steps 2 and 3.

\vspace{2mm}

 \begin{itemize}

\item \bf{Step 2: Uniform estimates.}
\end{itemize}

To begin with, we first claim 
that there exists a time $T_0>0$ such that for all $j\geq0$,
\begin{equation}\label{mmrn0}
\begin{aligned}
\sup\limits_{t\in(0,T_0]}\|c^j\|_{L^{\infty}}\leq  \|c_0\|_{L^{\infty}},
\end{aligned}
\end{equation}
and
\begin{equation}\label{mmrn}
\begin{aligned}&\sup\limits_{t\in(0,T_0]}\|(n^j, c^j,u^j)\|_{H^2}^2+ \int_0^{T_0}\|(\nabla n^{j},\nabla c^j,\Lambda^{\alpha}u^j)\|_{H^2}^2\,dt\leq 2\|(n_0, c_0,u_0)\|_{H^2}^2.
\end{aligned}
\end{equation}
In addition, we mention that the main difficulty lies in analyzing the high order nonlinear term 
$$
\nabla\cdot(\chi(c^j)n^j\nabla c^j),
$$
which requires the smallness of $L^2(0,T_0;H^2)$-norm of $\nabla c^j$. To achieve it, we let $c_{L}$ be the solution to the linear problem
\begin{align}
    &\partial_{t}c_{L}-\Delta c_{L}=0,\quad\quad c_{L}|_{t=0}=c_0.\nonumber
\end{align}
It is easy to verify that
\begin{equation}\nonumber
\begin{aligned}
&\sup_{t\in(0,T_0]}\|c_{L}\|_{H^2}^2+2\int_0^{T_0} \|\nabla c_{L}\|_{H^2}^2\,dt\leq \|c_0\|_{H^2}^2.
\end{aligned}
\end{equation}
This implies that $\mathop{\lim}\limits_{t\rightarrow0} \|\nabla c_{L}\|_{L^2(0,T;H^2)}=0$. Thus, for some small constant $R_0$ to be chosen later, one can find a time $T_{R_0}>0$ such that
\begin{equation}\label{mmrn1}
\begin{aligned}
\int_0^{T_{R_0}} \|\nabla c_{L}\|_{H^2}^2\,dt\leq R_0.
\end{aligned}
\end{equation}
We thus further claim that for all $j\geq0$ and $T_0\in (0,T_{R_0})$, the following uniform estimate for the error $c^{j}-c_{L}$ holds:
\begin{equation}\label{mmrn2}
\begin{aligned}
\int_0^{T_0} \|\nabla (c^{j}-c_{L})\|_{H^2}^2\,dt\leq R_0.
\end{aligned}
\end{equation}
In view of \eqref{mmrn1}-\eqref{mmrn2}, we have
\begin{equation}\label{mmrn3}
\begin{aligned}
\int_0^{T_0} \|\nabla c^{j} \|_{H^2}^2\,dt\leq 2R_0.
\end{aligned}
\end{equation}

Our goal is to find some small $T_0$ and $R_0$  such that \eqref{mmrn0}-\eqref{mmrn} and \eqref{mmrn2} indeed hold for all $j\geq 0$.  It is clear for $j=0,1$. We assume, by induction, that \eqref{mmrn0}-\eqref{mmrn} and \eqref{mmrn2}  hold for any fixed $j\geq 1$.

We now justify \eqref{mmrn0}-\eqref{mmrn} and \eqref{mmrn2}  for $j+1$.  First, one can achieve \eqref{mmrn0} using the standard maximum principle for the transport-diffusion $\eqref{NSn}_2$. By the standard energy estimates for the uniformly parabolic equations $\eqref{NSn}_1$- $\eqref{NSn}_3$, we obtain 
\begin{equation}\label{mrn}
\begin{aligned}
\frac{1}{2}\frac{d}{dt}&\|(n^{j+1}, c^{j+1},u^{j+1})\|_{H^2}^2+\|(\nabla n^{j+1},\nabla c^{j+1},\Lambda^{\alpha} u^{j+1})\|_{H^2}^2\\
&\leq \underbrace{\int_{\mathbb{R}^3}\sum\limits_{0\leq |\zeta|\leq 2}\partial^{\zeta}(-u^j\cdot \nabla n^{j}-
\nabla\cdot(\chi(c^j)n^{j} \nabla c^{j}))\partial^{\zeta}n^{j+1}\,dx}_{\mathcal{A}(t)}\\
&\quad+\underbrace{\int_{\mathbb{R}^3}\sum\limits_{0\leq |\beta|\leq 2}\partial^{\beta}(-u^j\cdot \nabla c^{j}-n^{j}f(c^j))\partial^{\beta} c^{j+1}\, dx}_{\mathcal{B}(t)}\\
&\quad+\underbrace{\int_{\mathbb{R}^3}\sum\limits_{0\leq |\gamma|\leq 2}\partial^{\gamma}(-u^{j} \cdot \nabla u^{j}-n^j\nabla \phi)\cdot \partial^{\gamma} u^{j+1}\,dx}_{\mathcal{C}(t)},
\end{aligned}
\end{equation}
where $\zeta=(\zeta_1,\zeta_2,\zeta_3)$,  $\beta=(\beta_1,\beta_2,\beta_3)$ and $\gamma=(\gamma_1,\gamma_2,\gamma_3)$. 

Lert $C$ denote some constant only depending on $\|c_0\|_{L^{\infty}}$, $f$, $\chi$ and $\nabla\phi$. Using \eqref{mmrn0} and the Gagliardo–Nirenberg-Sobolev inequalities, we obtain 
\begin{equation}\label{a}
\begin{aligned}
\mathcal{A}(t)\leq& \frac{1}{2}\|\nabla n^{j+1}\|_{H^2}^2+C\|n^{j+1}\|_{L^2}^2+C\|u^j\cdot\nabla n^j\|_{H^1}^2+C\|\nabla\cdot(\chi(c^j)n^j\nabla c^j)\|_{H^1}^2\\
\leq & \frac{1}{2}\|\nabla n^{j+1}\|_{H^2}^2+C\|u^j\|_{H^2}^2\|\nabla n^j\|_{H^1}^2+C(1+\|c^j\|_{H^2}^2)\|n^j\|_{H^2}^2 \|\nabla c^j\|_{H^2}^2,
\end{aligned}
\end{equation}
and
\begin{equation}\label{b}
\begin{aligned}
    \mathcal{B}(t)\leq& \frac{1}{2}\|\nabla c^{j+1}\|_{H^2}^2+C\|c^{j+1}\|_{L^2}^2+C\|u^j\cdot \nabla c^{j}\|_{H^1}^2+C\|n^jf(c^j)\|_{H^1}^2\\
    \leq & \frac{1}{2}\|\nabla c^{j+1}\|_{H^2}^2+\|u^j\|_{H^2}^2\|\nabla c^{j}\|_{H^2}^2+C\|n^j\|_{H^2}^2\|c^j\|_{H^1}^2.
\end{aligned}
\end{equation}
With regards to $\mathcal{C}(t)$, we consider two cases $\alpha\geq 1$ and  $\frac{1}{2}<\alpha<1$ separately. For $\alpha\geq 1$, it is direct to get
\begin{equation}\label{c1}
\begin{aligned}
    \mathcal{C}(t)\leq& \frac{1}{2}\|u^{j+1}\|_{H^2}^2+C\|u^j\cdot \nabla u^j\|_{H^1}^2+C\|n^j\nabla \phi\|_{H^2}^2+\frac{1}{2} \|\Lambda^{\alpha} u^{j+1}\|_{\dot{H}^2}^2\\
    \leq& \frac{1}{2}\|u^{j+1}\|_{H^2}^2+\frac{1}{2} \|\Lambda^{\alpha} u^{j+1}\|_{\dot{H}^2}^2+C\|u^j\|_{H^2}^4+C\|n^j\|_{H^2}^2.
\end{aligned}
\end{equation}
In the case $\frac{1}{2}<\alpha<1$, it follows from \eqref{G-GN} and \eqref{product0} that
\begin{equation}\label{c2}
    \begin{aligned}
  \mathcal{C}&(t)\leq \frac{1}{2}\|\Lambda^\alpha u^{j+1}\|_{\dot{H}^2}^2+\frac{1}{2}\|u^{j+1}\|_{H^2}^2+C\|u^j\cdot \nabla u^j\|_{H^1}^2+C\|n^j\nabla \phi\|_{H^2}^2\\
  &\quad +C\|u^j\otimes u^j\|_{\dot{H}^{3-\alpha}}^2\\
  &\leq \frac{1}{2}\|\Lambda^\alpha u^{j+1}\|_{\dot{H}^2}^2+\frac{1}{2}\|u^{j+1}\|_{H^2}^2+C\|u^j\|_{H^2}^4+C\|n^j\|_{H^2}^2\\
  &\quad+C\|u^j\|_{H^2}^2\|u^j\|_{\dot{H}^2}^{2\theta}\|\Lambda^{\alpha} u^j\|_{\dot{H}^2}^{2(1-\theta)},
    \end{aligned}
\end{equation}
where we used the fact  that, due to \eqref{G-GN} and \eqref{product0},
\begin{equation*}
\begin{aligned}
\|u^j\otimes u^j\|_{\dot{H}^{3-\alpha}}\leq& 2\|u^j\|_{L^\infty}\|u^j\|_{\dot{H}^{3-\alpha}}\leq C\|u^j\|_{H^2}\|u^j\|_{\dot{H}^{2}}^{\theta}\|\Lambda^{\alpha}u^j\|_{\dot{H}^{2}}^{1-\theta}
\end{aligned}
\end{equation*}
with $\theta=2-\frac{1}{\alpha}\in (0,1)$. 

Then, we substitute \eqref{a}-\eqref{c2} into \eqref{mrn}, and get 
\begin{equation*}
\begin{aligned}
\frac{d}{dt}&\|(n^{j+1},c^{j+1},u^{j+1}) \|_{H^2}^2+\|(\nabla n^{j+1},\nabla c^{j+1}, \Lambda^{\alpha} u^{j+1})\|_{H^2}^2\\
\leq & C\|(n^{j+1},c^{j+1},u^{j+1}) \|_{H^2}^2+C\|u^{j}\|_{H^2}^{2(1+\theta)}\|\Lambda^{\alpha} u^j\|_{\dot{H}^2}^{2(1-\theta)}\\
&+C\left(\|n^j\|_{H^2}^2+\|u^j\|_{H^2}^4+\|u^j\|_{H^2}^2\|n^j\|_{H^2}^2+(1+\|c^j\|_{H^2}^2)\|n^j\|_{H^2}^2 \|\nabla c_j\|_{H^2}^2\right),
\end{aligned}
\end{equation*}
for some $\widetilde{\theta}\in(0,1)$. Let $\mathcal{X}_0:=\|(n_0,c_0,u_0)\|_{H^2}^2$.  By virtue of Gr\"{o}nwall's inequality, it thence holds that
\begin{equation*}
\begin{aligned}
   \sup\limits_{t\in(0, T_0]}&\|(n^{j+1},c^{j+1},u^{j+1}) \|_{H^2}^2+\int_0^{T_0}\|(\nabla n^{j+1},\nabla c^{j+1}, \Lambda^{\alpha} u^{j+1})\|_{H^2}^2\, dt \\
\leq& e^{CT_0}\bigg(\mathcal{X}_0+C\sup_{t\in(0,T_0]}\|u^{j}\|_{H^2}^{2(1+\widetilde{\theta})}\|\Lambda^{\alpha}u^j\|_{L^2(0,T_0;H^2)}^{2(1-\widetilde{\theta})}T_0^{\theta}\\
&\quad+C(1+\sup_{t\in(0,T_0]}\|u^j\|_{H^2}^2)\sup_{t\in(0,T_0]}\|(n^j,u^j)\|_{H^2}^2T_0\\
&\quad+C(1+\sup_{t\in(0,T_0]}\|c^j\|_{H^2}^2)\sup_{t\in(0,T_0]}\|n^j\|_{H^2}^2\|\nabla c^j\|_{L^2(0,T_0;H^2)}^2\bigg),
\end{aligned}
\end{equation*}
which, together with \eqref{mmrn} and \eqref{mmrn3}, gives rise to
\begin{equation*}
\begin{aligned}
   \sup\limits_{t\in(0, T_0]}&\|(n^{j+1},c^{j+1},u^{j+1}) \|_{H^2}^2+\int_0^{T_0}\|(\nabla n^{j+1},\nabla c^{j+1}, \Lambda^{\alpha} u^{j+1})\|_{H^2}^2\, dt \\
   &\leq e^{CT_0} \bigg( \mathcal{X}_0+C\mathcal{X}_0^2 T_0^{\widetilde{\theta}}+C(1+\mathcal{X}_0)\mathcal{X}_0 T_0+(1+\mathcal{X}_0)\mathcal{X}_0 R_0\bigg)\\
   &\leq 2\|(n_0,c_0,u_0)\|_{H^2}^2,
\end{aligned}
\end{equation*}
provided that we first choose
$$
R_0:=\frac{1}{5(1+\mathcal{X}_0)\mathcal{X}_0},
$$
and then take
$$
T_0\leq T_1:=\min\left\{T_{R_0}, \log{\frac{5}{4}},  \frac{1}{(5C\mathcal{X}_0^{2})^{\frac{1}{\widetilde{\theta}}}},\frac{1}{5C\mathcal{X}_0(1+\mathcal{X}_0)}\right\}.
$$

In order to justify \eqref{mmrn2}, we consider the error equation
\begin{equation}\label{erroreq}
\begin{aligned}
&\partial_{t}(c^{j+1}-c_{L})-\Delta (c^{j+1}-c_{L})=-u^j\cdot \nabla c^j-n^j f(c^j),\quad\quad (c^{j+1}-c_{L})|_{t=0}=0.
\end{aligned}
\end{equation}
Performing $H^2$-energy estimates for \eqref{erroreq} and using \eqref{mmrn0}-\eqref{mmrn}, we can derive
\begin{equation}\nonumber
\begin{aligned}
&\sup_{t\in(0,T_0]}\|c^{j+1}-c_{L}\|_{H^2}^2+\int_{0}^{T_0}\|\nabla(c^{j+1}-c_{L})\|_{H^2}^2\,dt\\
&\leq \int_0^{T_0} \|u^j\cdot \nabla c^j+n^j f(c^j)\|_{H^1}^2\,dt\\
&\leq C \sup\limits_{t\in(0, T_0]}(\|u^j\|_{H^2}\|c^j\|_{H^2}+\|n^j\|_{H^2}\|c^j\|_{H^1})T_0\\
&\leq C \mathcal{X}_0 T_0\\
&\leq R_0,
\end{aligned}
\end{equation}
as long as we let
$$
T_0\leq T_2:= \frac{R_0}{C\mathcal{X}_0}.
$$
Hence, when $T_0\leq \min\{T_1,T_2\}$, we conclude that the uniform bounds \eqref{mmrn0}-\eqref{mmrn} and \eqref{mmrn3} hold true for any $j\geq0$.

In addition, taking $\div$ to $\eqref{NSn}_3$, we have 
$$
\Delta P^{j+1}=-\nabla u^j:\nabla u^{j}+\nabla\cdot(n^j\nabla \phi),
$$
from which we infer
\begin{equation}\label{pressure}
    \begin{aligned}
    \|\nabla P^{j+1}\|_{L^2(0,T_0;H^{2})}&\leq C\|\nabla u^j:\nabla u^{j}\|_{L^2(0,T_0;H^1)}+\|n^j\nabla \phi\|_{L^2(0,T_0;H^{2})}\\
    &\leq C\| \nabla u\|_{L^2(0,T_0;H^2)}\|\nabla u^j\|_{L^{\infty}(0,T;H^1)}+C\|n^j\|_{L^2(0,T_0;H^{2})}\\
    &\leq C(1+T_0)\|(n_0, c_0,u_0)\|_{H^2}^2.
    \end{aligned}
\end{equation}


\vspace{2mm}

\begin{itemize}
\item \bf{Step 3: Convergence and existence.}
\end{itemize}

In order to obtain the convergence of the approximate equations \eqref{NSn} to the original equations \eqref{eq1-1}, one needs to establish the strong compactness of the sequence $\{(n^j,c^j,u^j, P^j)\}_{j=0}^\infty$ in a suitable sense. To this end, we set 
\begin{equation}\nonumber
\begin{aligned}
W^j(t)&:=\|n^j-n^{j-1}\|_{L^2}^2+\|c^j-c^{j-1}\|_{H^1}^2+\| u^j-u^{j-1}\|_{L^2}^2,\\
X^j(t)&:=\|\nabla(n^j-n^{j-1})\|_{L^2}^2+\|\nabla(c^{j}-c^{j-1}) \|_{H^1}^2+\|\Lambda^\alpha(u^{j}-u^{j-1}) \|_{L^2}^2,\\
Y^j(t)&:=\|(n^j,c^j,u^j)\|_{H^2}^2.
\end{aligned}
\end{equation}
We take the difference between $j+1$ and $j$ for Eq.~\eqref{NSn}, and attain
\begin{equation}\nonumber
	\begin{cases}
\partial_t (n^{j+1}-n^j)-\Delta (n^{j+1}-n^j)=-u^j\cdot \nabla n^{j}+u^{j-1}\cdot \nabla n^{j-1}&\\
\quad\quad\quad\quad\quad\quad -
\nabla \cdot (\chi(c^j)n^{j} \nabla c^{j})+\nabla \cdot (\chi(c^{j-1})n^{j-1} \nabla c^{j-1}),&\\
\partial_t (c^{j+1}-c^j)-\Delta (c^{j+1}-c^j)=-u^j\cdot \nabla c^{j}-n^{j}f(c^j)+u^{j-1}\cdot \nabla c^{j-1}+n^{j-1}f(c^{j-1}),&\\
\partial_t(u^{j+1}-u^j)+(-\Delta)^\alpha (u^{j+1}-u^j)+\nabla (P^{j+1}-P^j)&\\
\quad\quad\quad\quad=-u^{j} \cdot \nabla u^{j}+u^{j-1} \cdot \nabla u^{j-1}-n^j\nabla \phi+n^{j-1}\nabla \phi,&\\
\nabla \cdot u^{j+1}=0.&\\	
	\end{cases}
\end{equation}
Similar to Step 2, we use standard energy estimates to obtain
\begin{equation}\label{salpha}
    \begin{aligned}
\frac{1}{2}\frac{d}{dt}W^{j+1}(t)+X^{j+1}(t)\leq& \frac{1}{2}W^{j+1}(t)+\frac{1}{2}X^{j+1}(t)
+CW^{j}(t)[Y^j(t)+Y^{j-1}(t)+1]\\
&+\|u^j\cdot\nabla u^j-u^{j-1}\cdot\nabla u^{j-1}\|_{\dot{H}^{-\alpha}}^2\textbf{1}_{\alpha\in(\frac{1}{2},\frac{3}{2})},
    \end{aligned}
\end{equation}
where 
\begin{equation*}
    \textbf{1}_{\alpha\in(\frac{1}{2},\frac{3}{2})}=\begin{cases}
        0,\quad \alpha\notin (\frac{1}{2},\frac{3}{2}),\\
        1,\quad \alpha\in (\frac{1}{2},\frac{3}{2}).
    \end{cases}
\end{equation*}
In \eqref{salpha}, the case of $\alpha\geq \frac{3}{2}$ can be directly controlled. However, one needs to handle the specific case of $\frac{1}{2}<\alpha<\frac{3}{2}$ due to the weaker dissipation. To this end, by means of  the product inequalities in Lemma~\ref{lemmaproduct}, we get 
\begin{equation*}
\begin{aligned}
&\|u^j\cdot\nabla u^j-u^{j-1}\cdot\nabla u^{j-1}\|_{\dot{H}^{-\alpha}}\\
&\leq \|(u^j-u^{j-1})\cdot\nabla u^j\|_{\dot{H}^{-\alpha}}+\|u^{j-1}\cdot\nabla( u^j-u^{j-1})\|_{\dot{H}^{-\alpha}}\\
&\leq \|u^i-u^{j-1}\|_{L^2}\|\nabla u^j\|_{\dot{H}^{\frac{3}{2}-\alpha}}+\|(u^j-u^{j-1})\|_{L^2}\|u^{j-1}\|_{\dot{H}^{\frac{5}{2}-\alpha}}\\
&\leq W^{j}(t)(Y^j(t)+Y^{j-1}(t)).
\end{aligned}
\end{equation*}
By means of the uniform estimates \eqref{mmrn0}-\eqref{mmrn} obtained in Step 2, we have 
\begin{equation}\label{step3}
\frac{d}{dt}W^{j+1}(t)\leq W^{j+1}(t)+C(1+\mathcal{X}_0) W^{j}(t),\quad\quad j\geq 1,
\end{equation}
where we recalled $\mathcal{X}_0:=\|(n_0,c_0,u_0)\|_{H^2}^2$. This, combined with Gr\"{o}nwall's inequality, yields 
\[
\sup\limits_{t\in(0,T_0]}W^{j+1}(t)\leq C(1+\mathcal{X}_0)e^{T_0}T_0\sup\limits_{t\in(0,T_0]}W^{j}(s).
\]
Therefore, we take
$$
T_0\leq \min\{T_1,T_2,T_3\}\quad\text{with}\quad T_3:=\min\left\{ \log 2, \frac{1}{4C(1+\mathcal{X}_0)}\right\}
$$
such that
\[
\sup\limits_{0\leq t \leq T}W^{j+1}(t)\leq\frac{1}{2}\sup\limits_{0\leq t \leq T}W^{j}(t),
\]
which means that there exists a limit $(n,c,u)$ such that as $j\rightarrow\infty$, $(n^j,c^j,u^j)$ converges to $(n,c,u)$ strongly in   in $L^{\infty}(0,T;L^2(\mathbb{R}^3))\times L^{\infty}(0,T;H^1(\mathbb{R}^3))\times L^{\infty}(0,T;L^2(\mathbb{R}^3))$. In addition, in light of \eqref{pressure}, there exists a limit $\nabla P$ such that , up to a subsequence,  $\nabla P^j$ converges to $\nabla P$ weakly in $L^2(0,T;H^2(\mathbb{R}^3))$ as $j\rightarrow \infty$. Thus, $(n,c,u)$ solves the system ~\eqref{eq1-1} in the sense of distributions on $[0,T_0]\times\mathbb{R}^3$. Thanks to the uniform estimates obtained in Step 2, Fatou's property implies 
\begin{equation}\nonumber
\begin{aligned}
\sup\limits_{t\in(0,T_0]}&\|(n,c,u)\|_{H^2} + \int_0^{T}\|(\nabla n,\nabla c,\Lambda^{\alpha} u\|_{H^2}^2\,dt\leq C\|(n_0,c_0,u_0)\|_{H^2}^2.
\end{aligned}
\end{equation}
Hence, $(n,c,u)$ is a strong solution to the Cauchy problem~\eqref{eq1-1}-\eqref{eq1-2}. Furthermore, by a standard argument one can show $\partial_{t}n,\partial_{t}c,\partial_{t}u\in L^2[0,T_0;H^1(\mathbb{R}^3))$. This together with Aubin-Lions Lemma implies that  $n,c,u\in C([0,T];H^2(\mathbb{R}^3))$.
\\
\vspace{2mm}

\begin{itemize}
\item \bf{Step 4: Uniqueness.}
\end{itemize}

It suffices to prove the uniqueness.  Let $(n^i,c^i,u^i)$, $i=1,2$ are two strong solutions to the Cauchy problem \eqref{eq1-1}-\eqref{eq1-2} satisfying \eqref{localr} for given time $T>0$ and supplemented with the same initial data. Then, similar to the calculations of \eqref{step3} in Step 3, we can obtain
\[\begin{aligned}
\frac{d}{dt}&(\|(n^1-n^2,u^1-u^2)\|_{L^2}^2+\|(c^1-c^2)\|_{H^1}^2)\\
\leq& C(1+\|(n^1,n^2,u^1,u^2)\|_{H^1}^2+\|(c^1,c^2)\|_{H^2}^2)(\|(n^1-n^2,u^1-u^2)\|_{L^2}^2+\|(c^1-c^2)\|_{H^1}^2)
\end{aligned} \]
for any $t\in(0,T]$. Thus, Gr\"onwall's lemma implies 
$$(n^1,c^1,u^1)(t,x)=(n^2,c^2,u^2)(t,x)\,\, \mbox{ for a.e. }\,\,(t,x)\in\mathbb{R}^3\times[0,T].$$
This completes the proof of Theorem \ref{Thlocal}.

\bigbreak
\bigbreak
\bigbreak


\vspace{2mm}

\textbf{Conflict of interest.} The authors do not have any possible conflict of interest.

\vspace{2mm}

\textbf{Data availability statement.}
 Data sharing not applicable to this article as no data sets were generated or analysed during the current study.

\bibliographystyle{abbrv} 
\parskip=0pt
\small
\bibliography{Reference}

\end{document}